\documentclass[12pt]{amsart}
  \RequirePackage{amsmath}
\RequirePackage{amssymb} \RequirePackage{amsthm}
\RequirePackage{epsfig} 
       \def\b{\beta}        
            \def\om{\omega} 
              
                  \def\z{\zeta}
         
\def\a{\alpha}
\def\ep{\varepsilon}

\def\z{\zeta}
\def\D{{\mathbb D}}  
  \def\N{{\mathbb N}}

\def\Xpa{{ X^p_{\alpha}}}

\def \Da {{\mathcal{D}_{\alpha}}}

\def\({\left(}       \def\){\right)}

\newcommand{\ig}{\stackrel{\text{def}}{=}}

\newcounter{capa}

\newtheorem{thm}{\sc Theorem}[section]

\newtheorem{prop}[thm]{\sc Proposition}
\newtheorem{lem}[thm]{\sc Lemma}
\newtheorem{mthm}{\sc Theorem}

\newtheorem{other}{\sc Theorem}              % Other papers' theorems
\newtheorem{otherl}[other]{\sc Lemma}
\numberwithin{equation}{section}
\begin{document}
\title
[Schatten classes  ] {Schatten classes of integration operators on
Dirichlet spaces}

\date{April 16, 2012}


%For each author, make a block with the following four macros:

\author[Jordi Pau]{Jordi Pau}
\address{Jordi Pau \\Departament de Matem\`{a}tica Aplicada i Analisi\\
Universitat de Barcelona\\
Gran Via 585 \\
08007 Barcelona\\
Spain} \email{jordi.pau@ub.edu}

\author{Jos\'e \'Angel Pel\'aez}
\address{Jos\'e \'Angel Pel\'aez\\ Departamento de An\'alisis
Matem\'atico\\
Universidad de M\'alaga\\ Campus de Teatinos\\ 29071 M\'alaga\\
Spain}\email{japelaez@uma.es}
%\urladdr{}
%%% ----------------------------------------------------------------------

\thanks{The first author is
 supported by SGR grant $2009$SGR $420$ (Generalitat de
Catalunya) and DGICYT grant MTM$2011$-$27932$-$C02$-$01$
(MCyT/MEC), while the second author is supported by: the Ram\'on y Cajal program
of MICINN (Spain), \lq\lq the Ministerio de Edu\-ca\-ci\'on y
Ciencia, Spain\rq\rq\, (MTM2011-25502) and from \lq\lq La Junta de
Andaluc{\'\i}a\rq\rq\, (FQM210), (P09-FQM-4468)}

\begin{abstract}
We address the question of describing the membership to Schatten-Von Neumann ideals $\mathcal{S}_ p$ of integration operators  $(T_ g f)(z)=\int_{0}^{z}f(\zeta)\,g'(\zeta)\,d\zeta$ acting on Dirichlet type spaces. We also study this problem for multiplication, Hankel and Toeplitz operators. In particular, we  provide   an extension of Luecking's result on Toeplitz  operators \cite[p. 347]{LJFA87}.
\end{abstract}
\maketitle

%%%%%%%%%%%%%%%%%%%%%%%%%%%
\section{Introduction and main results}
\par
  Let $\D$ denote the open unit disk of the complex
plane, and let $H(\D)$ be the class
of all analytic functions on $\D$. For $\alpha>-1$, let
\[dA_{\alpha}(z)=(\alpha+1)\,(1-|z|^2)^{\alpha}\,dA(z),\]
where $dA(z)=\frac{1}{\pi}\,dx\,dy$ is the normalized area measure on $\D$. For $\alpha\ge 0$, the weighted
Dirichlet-type space $\Da$ consists of those functions $f\in
H(\D)$ for which
\begin{displaymath}
\|f\|_{\Da}^2\ig
|f(0)|^2+\int_{\D}|f'(z)|^2\,dA_{\alpha}(z)<\infty,
\end{displaymath}
 Note that the space $\mathcal{D}_ 0$ is just the classical Dirichlet space and, as usual, will be simply denoted by $\mathcal{D}$. The spaces $\Da$ are reproducing kernel Hilbert spaces: for each $z\in \D$, there are functions $K^{\alpha}_ z\in \Da$ for which the reproducing formula $f(z)=\langle f,K^{\alpha}_ z\rangle _{\Da}$ holds, where the inner product in $\Da$ is given by
\[\langle f,g\rangle_{\Da} \ig f(0)\overline{g(0)}+\int_{\D}f'(z)\,\overline{g'(z)}\,dA_{\alpha}(z).\]
For $0<p<\infty$, we shall also write $A^p_\alpha$ for the weighted
 Bergman space of those $g\in H(\D)$ such that
 $$\|g\|^p_{A^p_\alpha}=\int_\D|g(z)|^p\,dA_{\alpha}(z)<\infty.$$

\par   Here we put our attention on the study of the integration operator $T_ g$ and the multiplication operator $M_g$
defined by
\begin{equation*}\begin{split}
&(T_ g f)(z)=\int_{0}^{z}f(\zeta)\,g'(\zeta)\,d\zeta,
\\ & M_g(f)=g(z)f(z).
\end{split}\end{equation*} where $g$ is an analytic function on $\D$.
The bilinear operator $\left(f,g\right)\rightarrow \int f\,g'$  was introduced by  A. Calder\'on in harmonic analysis in the $60$'s for his research on commutators of singular integral operators \cite{Calderon65} (see also \cite[p.1136]{SteinNotAmer}). After that, it and different variations going by the name of \lq\lq paraproducts\rq\rq, have been extensively studied, becoming fundamental tools in harmonic analysis.
Pommerenke was probably one of the first authors of the complex function theory community to consider the operator $T_g$ \cite{Pom}. After the pioneering works of Aleman and Siskakis \cite{AS0,AS}, the study of the operator $T_ g$ on several spaces of analytic functions
has attracted a lot of attention in recent years (see \cite{A,AC,PP,PelRat,SiS,SZ}).

\par Our main goal is to study the membership in the Schatten-Von Neumann ideals $\mathcal{S}_ p$ of the integration operator $T_ g:\Da\rightarrow \Da$.
  If $\alpha>1$, $\Da$ is nothing else but $A^2_{\alpha-2}$ and $\mathcal{D}_1=H^2$, the classical Hardy space, so   for $p>1$, then $T_g\in\mathcal{S}_ p(\Da)$ if and
only if $g$ belongs to the  Besov space $B_p$, and if
$0<p\le 1$, then $T_g\in \mathcal{S}_ p(\Da)$ if and only if $g$ is
constant (see \cite{AS0,AS}).
 We recall that, for $p>1$, the Besov space $B_p$ is the space of all analytic
functions $g$ in $\D$ such that
$$\int_{\D}|g'(z)|^p(1-|z|^2)^{p}\,d\lambda(z)<\infty,$$
where $d\lambda(z)=\frac{dA(z)}{(1-|z|^2)^2}$ is the hyperbolic measure on $\D$.
\par  The following result is implicit in the literature (see \cite{Wuark93}) and can be proved by  using the theory of Toeplitz operators (see Section \ref{operators}).

\begin{other}\label{th:1}
Let  $g\in H(\D)$. We have the following:
\begin{enumerate}
\item[(a)] Let $0<\alpha<1$ and $p>1$ with  $p(1-\alpha)<2$. Then $T_g\in S_p(\Da)$
if and only if $g$ belongs to  $B_p$.
\item[(b)]
If $0<p\le 1$ and $0<\alpha<1$, then
$T_g\in S_p(\Da)$ if and only if $g$ is constant.

\end{enumerate}
\end{other}
%This argument was also used by Z. Wu (see \cite{Wuark93}) in his study of Hankel operators on Dirichlet type spaces.
\medskip \par However for $0<\alpha<1$ and $p(1-\a)\ge 2$, to the best of our knowledge, it  is  an open problem founding a description of those $g\in H(\D)$ such that $T_g\in S_p(\Da)$.
This motivation leads us to introduce for $0\le \a<\infty$ and $1<p<\infty$, the space $X^p_{\alpha}$ which consists of those $g\in H(\D)$ such that
\begin{equation}\label{eq:xpa}
||g||^p_{X^p_{\alpha}} \ig |g(0)|^p +\int_{\D}\!\left (\!(1-|w|^2)^\alpha \!\!\int_{\D}\frac{|g'(z)|^2\, dA_{\a}(z)}{|1-\bar{w}z|^{2+2\a}} \right)^{p/2}\!\!\!\!(1-|w|^2)^{p-2}dA(w)<\infty.
\end{equation}

\par The following result gives a description of the membership in $\mathcal{S}_ p(\Da)$ in the range $p>1$ and $p(1-\alpha)<4$.

\begin{mthm}\label{th:nuevo}
Let $0<\alpha<1$, $g\in H(\D)$ and $p>1$ with  $p(1-\alpha)<4$. Then $T_g\in S_p(\Da)$
if and only if $g$ belongs to  $X^p_{\alpha}$.
\end{mthm}

\par Now we are going to deal with the case of the classical Dirichlet space $\mathcal{D}$. The situation here it seems to be more difficult. First of all, it is easy (and well known) to describe when the operator $T_ g$ belongs to the Hilbert-Schmidt class $\mathcal{S}_ 2(\mathcal{D})$. Indeed, for any orthonormal basis $\{e_ n\}$ of the Dirichlet space, one has (see Section \ref{s2})
\begin{equation}\label{HS}
\begin{split}
\|T_ g\|^2_{\mathcal{S}_ 2}=\sum_ n \|T_ ge_ n\|_{\mathcal{D}}^2&= \int_{\D} |g'(z)|^2 \sum_ n |e_ n(z)|^2\,dA(z)
\\
&=\int_{\D} |g'(z)|^2 \,\log \frac{e}{1-|z|^2}\,dA(z).
\end{split}
\end{equation}
Therefore, the integration operator $T_ g$ belongs to $\mathcal{S}_ 2(\mathcal{D})$ if and only if the last integral in the previous equation is finite. The class of functions $g\in H(\D)$ satisfying this condition shall be denoted by $\mathcal{DL}$.
\par  If $1<p<2$ Theorem \ref{th:1} suggests that the membership in $\mathcal{S}_ p(\mathcal{D})$ of the operator $T_ g$ could be described by  those $g$ being in the Besov space $B_ p$. However, since for $p<2$ any operator on $\mathcal{S}_ p$ must be Hilbert-Schmidt, clearly the condition $g\in \mathcal{DL}$ is necessary for $T_ g$ being in $\mathcal{S}_ p(\mathcal{D})$, and an easy calculation shows that the function $g(z)=\log \log \frac{e}{1-z}$ belongs to $B_ p$ for all $p>1$ but $g$ is not in $\mathcal{DL}$. Thus, the  condition $g\in B_ p$ is not sufficient to assert that $T_ g$ is in $\mathcal{S}_ p(\mathcal{D})$.
\par On the other hand, as in the weighted case, there are no trace class integration operators in the Dirichlet space unless $g$ is constant.
\begin{mthm}\label{th:2}
Let $0<p\le 1$ and $g\in H(\D)$. Then $T_ g \in \mathcal{S}_ p(\mathcal{D})$ if and only if $g$ is constant.
\end{mthm}
For the case $1<p<2$ we have a necessary condition and a different sufficient condition. We will see that they are sharp in a certain sense. Before that, for $p>1$ and $\gamma>0$, we consider the space
 $B_{p,\log^{\gamma}}$, that consists of those functions $g$ analytic on $\D$
 such that
\[ \|g\|^p_{B_{p,\log^{\gamma}}}=|g(0)|^p+\int_\D|g'(z)|^p
\left(\log\frac{e}{1-|z|}\right)^{\gamma}(1-|z|^2)^{p-2}\,dA(z)<\infty.\]

\begin{mthm}\label{th:3}
Let $1<p<2$ and $g\in H(\D)$. Then
\begin{enumerate}
\item[(a)] If $T_ g\in \mathcal{S}_ p(\mathcal{D})$, then $g\in B_ p$.

\item[(b)] If $g\in B_{p,\log^{p/2}}$, then $T_ g\in \mathcal{S}_ p(\mathcal{D})$. Moreover, $\|T_ g\|_{\mathcal{S}_ p}\le C \|g\|_{B_{p,\log^{p/2}}}$.

\item[(c)] If $g\in X^p_{0}$, then $T_ g\in \mathcal{S}_ p(\mathcal{D})$.
\end{enumerate}
\end{mthm}
When one takes the monomials as the symbols, it turns out that the correct behavior of $\|T_ g\|_{\mathcal{S}_ p}$ is given by $B_ p$ or $X^p_{0}$, while if one takes as a symbol to be functions of the type $g_ a(z)=(1-\bar{a}z)^{-\gamma}$, the correct behavior is   given by the $B_{p,\log^{p/2}}$ condition (see Lemmas \ref{le:normagkxpa} and \ref{le:ga}).
\par The case $p>2$ seems to be a mystery. Let $\mathcal{D}^p_{\beta}$ denote the space of those functions $f$ with $f'\in A^p_{\beta}$. For $p>2$, the inclusion $\mathcal{D}^p_{\beta} \subset \mathcal{D}$ holds if and only if $\beta<(p-2)/2$; and $\mathcal{D}\subset \mathcal{D}^p_{\beta}$ if and only if $\beta \ge p-2$ (see \cite[p.94]{ZZ}). Thus, if one is looking for conditions on the integrability of $g'$, it can not be expected some necessary condition much better than $B_ p=\mathcal{D}^p_{p-2}$, and a sufficient condition must be stronger than $g$ being in $\mathcal{D}^p_{\frac{p-2}{2}}$. We will discuss a little bit this case in Section \ref{sclassical}.

We close this section saying that from now on the paper is organized as
follows. In Section \ref{s2} we introduce several preliminary general results related on Schatten classes of operators on Dirichlet spaces. Section \ref{smain} is devoted to the proof of Theorem \ref{th:nuevo}. There it  will be proved directly (see Proposition \ref{pr:xpaproperties} (iv)) the identity
\begin{equation}\label{eq:BPC}
X^p_\a=B_p,\quad p>1,\, \alpha>0,\,\text{and $p(1-\alpha)<2$},
\end{equation}
which together with Theorem \ref{th:nuevo}
gives a proof of Theorem \ref{th:1} not relying in the theory of Toeplitz operators. It is worth mentioning that the Besov space $B_ p$ is rich of several characterizations (the identity \eqref{eq:BPC} gives a new one), each of them being the appropriate tool to use in different situations (see \cite{AFP}, \cite{BP}, or \cite{ZhuJMAA91} for example).  In Section \ref{sclassical}  we prove Theorem \ref{th:2} and Theorem \ref{th:3}. Also, by using some testing classes of functions,  we show that those results are sharp in a certain sense. Finally, Section \ref{operators} is devoted to study the relationship of the integration operator $T_ g$ with other classical operators acting on weighted Dirichlet spaces, such as Toeplitz operators, multiplication operators or big and small Hankel operators.  A similar connection also happens in other contexts \cite{PottSmithJFA2004}.
 Indeed, the same techniques used in the proof of Theorem \ref{th:nuevo}  work to demonstrate an  extension for positive Borel measures of the helpful  result of Luecking on Toeplitz operators \cite[p. $347$]{LJFA87}).

\par
Throughout the paper, the letter $C$ will denote a positive absolute
constant whose value may change at different occurrences, and we write $A\asymp B$ when the two quantities $A$ and $B$ are comparable.
 \section{Preliminary results}\label{s2}
 \par Let $H$ and $K$ be separable Hilbert spaces. Given $0<p<\infty$, let $\mathcal{S}_ p(H,K)$  denote the
Schatten $p$-class of operators from $H$ to $K$. If $H=K$ we simply
shall write $\mathcal{S}_ p(H)$. The class $\mathcal{S}_ p(H,K)$
consists of those compact operators $T$ from $H$ to $K$  with its sequence of singular
numbers $\lambda_ n$ belonging to $\ell^p$, the $p$-summable
sequence space. We recall that the singular numbers of a
compact operator $T$ are the square root of the eigenvalues of the
positive operator $T^*T$, where $T^*$ denotes the Hilbert adjoint of $T$. We remind the reader that $T\in \mathcal{S}_p(H)$ if and only
if $T^* T\in \mathcal{S}_{p/2}(H)$.
Also, the compact operator $T$ admits a decomposition of the form
\[T=\sum_ n \lambda_ n \, \langle \cdot, e_ n\rangle_ H \,\sigma_ n,\] where
$\{\lambda_ n\}$ are the singular numbers of $T$, $\{e_ n\}$ is an
orthonormal set in $H$, and $\{\sigma_ n\}$ is an orthonormal set in
$K$.
\par For $p\ge 1$, the class $\mathcal{S}_ p(H,K)$ is a Banach space equipped
with the norm
\[\|T\|_ {\mathcal{S}_ p}=\left (\sum_ n |\lambda_ n|^p\right
)^{1/p},\] while for $0<p<1$ one has the inequality $\|S+T\|_{\mathcal{S}_
p}^p\le \|S\|_{\mathcal{S}_ p}^p+\|T\|_{\mathcal{S}_ p}^p.$ We refer to \cite{Sim} or \cite[Chapter 1]{Zhu}
for a brief account on the theory of Schatten $p$-classes.
\par We shall write $H$ for a Hilbert
space of analytic functions in $\D$ with reproducing kernels
$K_z$.  Given an operator $T$  on $H$, usually  the reproducing kernel functions carry a large amount of information about relevant properties
of $T$, such as boundedness, compactness, membership in Schatten $p$-classes, etc.
 It is known that if $\{e_n\}$ is an orthonormal basis of a Hilbert
space $H$ of analytic functions in $\D$ with reproducing kernel
$K_z$, then
    \begin{equation}\label{RKformula}
    K_z(\zeta)=\sum_n e_ n(\zeta)\,\overline{e_ n(z)}
    \end{equation}
for all $z$ and $\zeta$ in $\D$, see e.g.~\cite[Theorem
$4.19$]{Zhu}. We also introduce $J_z$, the derivative of $K_z$ respect to $\overline{z}$, that is,
 \begin{equation}\label{RKdotformula}
    J_z(\zeta)=\sum_n e_ n(\zeta)\,\overline{e'_ n(z)}.
    \end{equation}
 It follows that
    \begin{equation}\label{eqRK1}
    \sum_n|e_n(z)|^2\le\|K_ z\|_{H}^2,\quad \text{and}\quad \sum_n|e'_n(z)|^2\le\|J_z\|_{H}^2
    \end{equation}
for any orthonormal set $\{e_n\}$ of $H$, and equality in
\eqref{eqRK1} holds if $\{e_ n\}$ is an orthonormal basis of $H$.
We shall write $k_z$ and $j_z$ for the normalizations of these functions.
 \par
 In order to avoid some confusions when dealing with reproducing kernels of either $\Da$ or $A^2_{\alpha}$, we use the notation $B_{z}^{\alpha}$ for the reproducing kernel of the weighted Bergman space $A^2_{\alpha}$ at the point $z$,
  and let
  $b_{z}^{\alpha}=\frac{B_{z}^{\alpha}}{\|B_{z}^{\alpha}\|_{A^2_\alpha}}$ be its normalization. It is well known (see \cite[Corollary $4.20$]{Zhu}) that
\begin{equation}\label{RKBergman}
B_ z^{\alpha}(w)=\frac{1}{(1-\bar{z}{w})^{2+\alpha}},\quad \textrm{ and}\quad \|B_ z^{\alpha}\|_{A^2_{\alpha}}=(1-|z|^2)^{-\frac{(2+\alpha)}{2}}.
\end{equation}
The reproducing kernel function for the Dirichlet type space $\Da$ is denoted by
$K^{\alpha}_ z$, and $k_ z^{\alpha}$ denotes the corresponding normalized reproducing kernel. Since $f\in \Da$ if and only if $f'\in A^2_{\alpha}$, using the reproducing formula for the Bergman space $A^2_{\alpha}$ (see \cite[Proposition $4.23$]{Zhu}), it can be deduced the following expression of the reproducing kernel of $\Da$ (see \cite{BP} or \cite{Wuark93}):
\begin{equation}\label{FRK}
  K^{\alpha}_ z(w)=1+\int_{0}^{w}\int_{0}^{\bar{z}}\frac{d\zeta}{(1-\eta\zeta)^{2+\alpha}}\,d\eta.
\end{equation}
In particular, for $\alpha=0$,
\begin{displaymath}
K^{\mathcal{D}}_ z(w):=K^{0}_ z(w)=1+\log \frac{1}{1-\bar{z}w}.
\end{displaymath}
Also, it is easy to see that
\begin{equation}\label{RKDA}
\|K^{\alpha}_ z\|_{\Da}^2= K_ z^{\alpha}(z)\asymp
\left \{
\begin{array}{ccc}
\log \frac{e}{1-|z|^2} & \textrm{if} & \alpha=0
\\
(1-|z|^2)^{-\alpha} & \textrm{if} &\alpha>0
\end{array}\right .
.
\end{equation}
 \par The next two results are certainly well known to the experts (see \cite{HS} or \cite{Msm} for similar results), but we find convenient for the reader to give a proof here.

\begin{prop}\label{th:general}

Let $T:A^2_{\alpha}\rightarrow H$ be a compact operator, where $H$
is any separable Hilbert space.
\begin{enumerate}
\item[\rm(i)]  For $p\ge 2$,
\[\int_{\D} \|T b_{z}^{\alpha}\|_ H ^p\,d\lambda (z)\le \frac{1}{1+\alpha}\,
\|T\|_{\mathcal{S}_ p}^p.\]

\item[\rm(ii)] For $0<p\le 2$,
\[\|T\|_{\mathcal{S}_ p}^p\le (1+\alpha)\int_{\D} \|T b_{z}^{\alpha}\|_ H ^p\,d\lambda
(z).\]
\end{enumerate}
\end{prop}
\begin{proof}
Since the operator $T$ is compact, it admits the decomposition
\[Tf=\sum_ n \lambda_ n \langle f,e_ n\rangle_{A^2_{\alpha}} f_
n,\] where $\{\lambda_ n\}$ are the singular values of $T$, $\{e_
n\}$ is an orthonormal set in $A^2_{\alpha}$, and $\{f_ n\}$ is an
orthonormal set in $H$. Then
\[TB_ z^{\alpha}=\sum_ n \lambda_ n \overline{e_ n(z)} f_ n,\]
and therefore
\[\|TB_ z^{\alpha}\|_{H}^2=\sum_ n |\lambda_ n|^2\,|e_ n(z)|^2.\]
Now, if $p\ge 2$, using H\"{o}lder's inequality, \eqref{eqRK1} and \eqref{RKBergman}, yields
\begin{displaymath}
\begin{split}
\int_{\D} \|T b_{z}^{\alpha}\|_ H ^p\,d\lambda (z)&=\int_{\D} \|T
B_{z}^{\alpha}\|_ H ^p\,\|B_ z^{\alpha}\|_{A^2_{\alpha}}^{-p}\,d\lambda (z)
\\
&=\int_{\D} \left (\sum_ n |\lambda_ n|^2\,|e_ n(z)|^2\right
)^{p/2}\,\|B_ z^{\alpha}\|_{A^2_{\alpha}}^{-p}\,d\lambda (z)
\\
&\le \int_{\D} \left (\sum_ n |\lambda_ n|^p\,|e_ n(z)|^2\right
)\left (\sum_ n |e_ n(z)|^2 \right )^{\frac{p-2}{2}} \!\!\|B_
z^{\alpha}\|_{A^2_{\alpha}}^{-p} \,d\lambda(z)
\\
&\le \sum_ n |\lambda_ n|^p \int_{\D}|e_ n(z)|^2\,\|B_
z^{\alpha}\|_{A^2_{\alpha}}^{-2} \,d\lambda(z)
\\
&=\sum_ n |\lambda_ n|^p \int_{\D} |e_ n(z)|^2
(1-|z|^2)^{\alpha}\,dA(z)=\frac{1}{1+\alpha}\,\|T\|_{S_ p}^p.
\end{split}
\end{displaymath}
If $0<p\le 2$, a similar argument, using H\"{o}lder's inequality with exponent $2/p\ge 1$, \eqref{eqRK1}  and \eqref{RKBergman}, gives
\begin{displaymath}
\begin{split}
\|T\|_{S_ p}^p&=(1+\alpha)\int_{\D} \sum_ n |\lambda_ n|^p \,|e_ n(z)|^2\,\|B_
z^{\alpha}\|_{A^2_{\alpha}}^{-2} \,d\lambda(z)
\\
& \le (1+\alpha)\int_{\D}\left (\sum_ n |\lambda_ n|^2\,|e_ n(z)|^2\right
)^{\frac{p}{2}}\! \left (\sum_ n |e_ n(z)|^2 \right )^{\frac{2-p}{2}}\!\!\|B_
z^{\alpha}\|_{A^2_{\alpha}}^{-2} \,d\lambda(z)
\\
&\le (1+\alpha)\int_{\D}\left (\sum_ n |\lambda_ n|^2\,|e_ n(z)|^2\right
)^{\frac{p}{2}} \!\|B_ z^{\alpha}\|_{A^2_{\alpha}}^{-p} \,d\lambda(z)
\\
&=(1+\alpha)\int_{\D} \|T b_{z}^{\alpha}\|_ H ^p\,d\lambda (z).
\end{split}
\end{displaymath}
\end{proof}
The corresponding analogue of Proposition \ref{th:general} for the Dirichlet type spaces $\Da$ uses the functions $j^{\alpha}_ z\ig\frac{J^{\alpha}_ z}{\| J^{\alpha}_ z\|_{\Da}}$.
\begin{prop}\label{pr:gen}

Let $T:\Da\rightarrow H$ be a compact operator, where $H$
is any separable Hilbert space.
\begin{enumerate}
\item[\rm(i)]  For $p\ge 2$,
\[\int_{\D} \|T j^{\alpha}_z\|_ H ^p\,d\lambda (z)\le \frac{1}{1+\alpha}
\|T\|_{\mathcal{S}_ p}^p.\]

\item[\rm(ii)] For $0<p\le 2$,
\[\|T\|_{\mathcal{S}_ p}^p\le \|T\|^p+(1+\alpha)\!\int_{\D} \|T j^{\alpha}_z\|_ H ^p\,d\lambda
(z).\]
\end{enumerate}
\end{prop}
\begin{proof}
Since $T$ is compact, it admits the decomposition
$$Tf=\sum_ n \lambda_ n \langle f,e_ n\rangle_{\Da} f_
n,$$ where $\{\lambda_ n\}$ are the singular values of $T$, $\{e_
n\}$ is an orthonormal set in $\Da$, and $\{f_ n\}$ is an
orthonormal set in $H$. It follows from \eqref{FRK} that $J^{\alpha}_z(0)=0$, then using \eqref{RKBergman},
\begin{equation}\label{Id2}
\|J^{\alpha}_z\|_{\Da}=\|B_ z^{\alpha}\|_{A^2_{\alpha}}=(1-|z|^2)^{-\frac{(2+\alpha)}{2}},
\end{equation}
and
\begin{displaymath}
\langle J^{\alpha}_z ,e_ n\rangle_{\Da}=\langle B_ z^{\alpha},e'_ n\rangle_{A^2_{\alpha}}=\overline{e'_ n(z)}.
\end{displaymath}
Thus
$T J^{\alpha}_z=\sum_ n \lambda_ n \overline{e'_ n(z)} f_ n,$
and therefore
\begin{equation}\label{Id1}
\|TJ^{\alpha}_z\|_{H}^2=\sum_ n |\lambda_ n|^2\,|e'_ n(z)|^2.
\end{equation}
If $p\ge 2$, using the identity \eqref{Id1}, H\"{o}lder's inequality, \eqref{eqRK1} and \eqref{Id2}
\begin{displaymath}
\begin{split}
\int_{\D} \|T j^{\alpha}_z\|_ H ^p\,d\lambda (z)
&=\int_{\D} \left (\sum_ n |\lambda_ n|^2\,|e'_ n(z)|^2\right
)^{p/2}\!\!\|J^{\alpha}_z\|_{\Da}^{-p}\,d\lambda (z)
\\
&\le \int_{\D} \left (\sum_ n |\lambda_ n|^p\,|e'_ n(z)|^2\right
)\left (\sum_ n |e'_ n(z)|^2 \right )^{\frac{p-2}{2}}\!\! \|J^{\alpha}_z\|_{\Da}^{-p} d\lambda(z)
\\
&\le \sum_ n |\lambda_ n|^p \int_{\D}|e'_ n(z)|^2\,\|J^{\alpha}_z\|_{\Da}^{-2} \,d\lambda(z)
\\
&=\sum_ n |\lambda_ n|^p \int_{\D} |e'_ n(z)|^2
(1-|z|^2)^{\alpha}\,dA(z)\le \frac{1}{1+\alpha}\,\|T\|_{\mathcal{S}_ p}^p.
\end{split}
\end{displaymath}
If $0<p\le 2$, since $\|e_ n\|_{\Da}=1$, and $dA_{\alpha}(z)=(1+\alpha)\|J^{\alpha}_z\|_{\Da}^{-2} \,d\lambda(z)$ due to \eqref{Id2}, then
\begin{equation}\label{Tp2}
\begin{split}
\|T\|_{\mathcal{S}_ p}^p&= \sum_ n |\lambda_n|^p\,|e_ n(0)|^2 \! + \!(1+\alpha)\!\sum_ n |\lambda_ n|^p \!\!\int_{\D}|e'_ n(z)|^2\,\|J^{\alpha}_z\|_{\Da}^{-2} \,d\lambda(z)
\\
&=(I)+(II).
\end{split}
\end{equation}
For the first term (I), observe that $|\lambda_ n|\le \|T\|$, and therefore
\begin{displaymath}
(I)\,\le \,\|T\|^p\sum_ n |e_ n(0)|^2=\|T\|^p \,\|K_ 0^{\alpha}\|^2_{\Da}= \|T\|^p.
\end{displaymath}
For the second term (II), due to H\"{o}lder's inequality, \eqref{eqRK1} and the identity \eqref{Id1}
\begin{displaymath}
\begin{split}
(II)& \,\le (1+\alpha)\int_{\D}\left (\sum_ n |\lambda_ n|^2\,|e'_ n(z)|^2\right
)^{p/2} \!\! \left (\sum_ n |e'_ n(z)|^2 \right )^{\frac{2-p}{2}}\!\!\|J^{\alpha}_z\|_{\Da}^{-2} \,d\lambda(z)
\\
&\le (1+\alpha)\int_{\D}\left (\sum_ n |\lambda_ n|^2\,|e'_ n(z)|^2\right
)^{p/2}\! \!\|J^{\alpha}_z\|_{\Da}^{-p} \,d\lambda(z)
\\
&=(1+\alpha)\int_{\D} \|T j^{\alpha}_z\|_ H ^p\,d\lambda (z).
\end{split}
\end{displaymath}
Putting the estimates obtained for (I) and (II) in \eqref{Tp2} we obtain part (ii). This completes the proof.
\end{proof}
\par The following result will also  be needed.
\begin{lem}\label{le:2}
Let $\alpha\ge 0$. If  $1\le p<2$ there is a  constant $C=C(p,\alpha)>0$ such that
\begin{displaymath}
\sum_ n |e_n(z)|^p|e'_n(z)|^{2-p}\ge C \, (1-|z|^2)^{p-2-\alpha},\quad |z|\to 1^-
\end{displaymath}
for any orthonormal basis $\{e_n\}$ of $\Da$.
\end{lem}
\begin{proof}
\par Let $\{e_n\}$ be any orthonormal basis of $\Da$. From \eqref{RKdotformula} and \eqref{FRK} we have
\begin{displaymath}
\sum_ n \overline{e'_ n(z)}
\,e_ n(z)=J^{\alpha}_{z}(z)=\int_{0}^{z} \frac{d\eta}{(1-\bar{z}\eta)^{2+\alpha}}=\frac{1-(1-|z|^2)^{1+\alpha}}{(1+\alpha)\,\bar{z}\,(1-|z|^2)^{1+\alpha}},
\end{displaymath}
and, since $\alpha\ge 0$, we obtain
\begin{displaymath}
\begin{split}
|z| \,(1-|z|^2)^{-1-\alpha} &\le (1+\alpha)\,\sum_ n |e_ n(z)|\,|e'_ n(z)|,
\end{split}
\end{displaymath}
which gives the result for $p=1$. If $1<p<2$, using H\"{o}lder's inequality
\begin{displaymath}
\begin{split}
\frac{|z|}{1+\alpha}\,(1-|z|^2)^{-1-\alpha} &\le \sum_ n |e_ n(z)|\,|e'_ n(z)|
\\
&\le \left (\sum_ n |e_ n(z)|^p\,|e'_ n(z)|^{2-p}\right )^{1/p}\left
(\sum_ n |e'_ n(z)|^2\right )^{1/p'}
\\
&\le C\left (\sum_ n |e_ n(z)|^p\,|e'_ n(z)|^{2-p}\right )^{1/p}
(1-|z|^2)^{-(2+\alpha)/p'},
\end{split}
\end{displaymath}
where the last inequality follows from \eqref{eqRK1} and \eqref{Id2}.
From here one obtains the corresponding inequality. The proof is
complete.
\end{proof}

We shall also use several times the following integral estimate (see \cite{Zhu}) that has become
indispensable in this area of analysis.
\begin{otherl}\label{Ict}
Suppose $z\in\D$, $c\ge 0$  and $t>-1$. The integral
$$I_{c,t}(z)=\int_{\D}\frac{(1-|w|^2)^t}{|1-\bar{w}z|^{2+t+c}}\,dA(w)$$
is comparable to $(1-|z|^2)^{-c}$ if $c>0$, and to $\log \frac{1}{1-|z|^2}$ if $c=0$.
\end{otherl}

\par The  useful inequality which appears below is from \cite{of}, and can be thought as a generalized version of the previous one.
\begin{otherl}\label{LI2}
Let $s>-1$, $r,t>0$, and $r+t-s>2$. If $t<s+2<r$ then, for $a,z\in \D$, we have
\begin{displaymath}
\int_{\D}\frac{(1-|w|^2)^s}{|1-\bar{w}z|^r\,|1-\bar{w}a|^t}\,dA(w)\leq
C\,\frac{(1-|z|^2)^{2+s-r}}{|1-\bar{a}z|^t}.
\end{displaymath}
\end{otherl}
\mbox{}
\par
For $z\in \D$ and
$r>0$, let
$$D(z,r)=\{w\in \D:\beta(z,w)<r\}$$ denote the hyperbolic disk
with center $z$ and radius $r$. Here $\beta(z,w)$ is the Bergman
or hyperbolic metric on $\D$.
\par We also need the concept of an $r$-lattice in the Bergman metric.
Let $r>0$. A sequence $\{a_ k\}$ of points in $\D$ is called an
$r$-lattice, if the unit disk is covered by the Bergman metric
disks $\{D_ k:=D(a_ k, r)\}$, and $ \beta(a_ i, a_ j)\ge r/2$ for all
$i$ and $j$ with $i \ne j$. If $\{a_ k\}$ is an $r$-lattice in
$\D$, then it also has the following property:  for any $R > 0$
there exists a positive integer $N$ (depending on $r$ and $R$)
such that every point in $\D$ belongs to at most $N$ sets in
$\{D(a_ k,R)\}$. There are elementary constructions of
$r$-lattices in $\D$. See \cite[Chapter 4]{Zhu} for example.

\section{Case $0<\alpha<1$.}\label{smain}
Before embarking on the proof of Theorem \ref{th:nuevo}, some preliminary results of interest on their own must be proved.
\subsection{A new class of spaces}
\par In this subsection, we display several nesting properties of $\Xpa$ and $B_p$  spaces.
 We offer a proof of \eqref{eq:BPC}, which gives under those restrictions an equivalent $B_p$-norm. It is worth noticing that  equivalent and useful $B_p$-norms (see \cite{AFP} and \cite{BP} for example) have been previously introduced  for the study of  operators on different spaces of analytic functions on $\D$. Also, our next result proves that $\Xpa\subsetneq B_p$ if $0<\alpha<1$ and  $p(1-\alpha)\ge 2$.
In fact,
  $B_p\subset \Da$ if $p(1-\alpha)<2$, and this is no
longer true when $p(1-\alpha)\ge 2$.
 % that $\Xpa\subsetneq B_p$ if $0<\alpha<1$ and  $p(1-\alpha)\ge 2$.
\begin{prop}\label{pr:xpaproperties}
Let $1<p<\infty$ and $\alpha\ge 0$. Then
\begin{enumerate}
\item[\rm(i)] $\Xpa\subset \Da \cap B_ p$.
\item[\rm(ii)] If $p<q$, then $\Xpa\subset X^q_\a$.
\item[\rm(iii)] If $0\le \alpha<\gamma$, then $\Xpa\subset X^p_{\gamma}$.
\item[\rm(iv)]
Let  $\alpha>0$. If  $p(1-\alpha)<2$ then $\Xpa=B_ p$.
\end{enumerate}
\end{prop}
\begin{proof}
For $a\in \D$ fixed, let $D(a):=\big\{z: |z-a|<\frac{1-|a|}{2}\big\}$.
\par (i)\,If $g\in \Xpa$, then the subharmonicity of $|g'|^2$ together with the fact that $|1-\bar{w}z|\asymp (1-|w|^2)$ for $z\in D(w)$ implies that $g\in B_ p$. Also, since $|1-\bar{w}z|\le 2$,
\begin{displaymath}
\begin{split}
||g||^p_{\Xpa}&\ge 2^{-(1+\alpha)p} \int_{\D} \left (\int_{\D} |g'(z)|^2\,dA_{\alpha}(z)\right)^{p/2}(1-|w|^2)^{p-2+\frac{\alpha p}{2}}\,dA(w)
\\
&=C_{p,\alpha}\|g'\|^p_{A^2_{\alpha}}.
\end{split}
\end{displaymath}
This shows that $\Xpa\subset \Da$ proving (i).

\par (ii) Assume that $g\in \Xpa$. Fix $a\in \D$. If $w\in D(a)$, then $(1-|w|)\asymp (1-|a|)$ and $|1-\bar{w}z|\asymp |1-\bar{a}z|$ for $z\in \D$. So
\begin{displaymath}
\begin{split}
||g||^p_{\Xpa}&\ge \int_{D(a)}\left (\int_{\D} \frac{|g'(z)|^2\,dA_{\alpha}(z)}{|1-\bar{w}z|^{2+2\alpha}}\right )^{p/2} (1-|w|^2)^{p-2+\alpha p/2}\,dA(w)
\\
&\ge C (1-|a|^2)^{p+\alpha p/2}\left (\int_{\D} \frac{|g'(z)|^2\,dA_{\alpha}(z)}{|1-\bar{a}z|^{2+2\alpha}}\right )^{p/2}.
\end{split}
\end{displaymath}
This gives
\begin{equation}\label{sup}
\sup_{a\in \D} (1-|a|^2)^{2+\alpha}\int_{\D} \frac{|g'(z)|^2\,dA_{\alpha}(z)}{|1-\bar{a}z|^{2+2\alpha}}<\infty,
\end{equation}
and  it follows easily that $||g||^q_{X^q_\a}\le C ||g||^p_{\Xpa}$ for $q>p$.
\par (iii) follows from the inequality $\sup_{z\in\D}\frac{(1-|z|^2)(1-|w|^2)}{|1-\overline{w}z|^2}\le 1$.
\par (iv)  The inclusion $X^p_\alpha\subset B_p$ follows from  (i).
\par Conversely, suppose that $g\in B_ p$. Assume first that $p>2$.
 Since $p\a>p-2$, we can choose $\varepsilon>0$ with $p\a-(1+\varepsilon)(p-2)>0$. Then, using H\"{o}lder's inequality and Lemma \ref{Ict}, we obtain
\begin{displaymath}
\begin{split}
\left (\int_{\D}\frac{|g'(z)|^2 dA_{\a}(z)}{|1-\bar{w}z|^{2+2\a}}\right)^{p/2}&\le \left (\int_{\D} \frac{|g'(z)|^p dA_{t}(z)}{|1-\bar{w}z|^{2+p\a}}\right )\!
\left (\int_{\D}\frac{(1-|z|^2)^{-1+\varepsilon} dA(z)}{|1-\bar{w}z|^{2}}\right)^{\frac{p-2}{2}}
\\
&
\le C\left (\int_{\D} \frac{|g'(z)|^p dA_{t}(z)}{|1-\bar{w}z|^{2+p\a}}\right )
\,(1-|w|^2)^{(-1+\varepsilon)\frac{p-2}{2}},
\end{split}
\end{displaymath}
where $t=\frac{(1-\varepsilon)(p-2)+\a p}{2}$. This gives
\begin{displaymath}
||g-g(0)||^p_{X^p_\a}\le C \int_{\D} |g'(z)|^p \left (\int_{\D} \frac{(1-|w|^2)^{\beta}}{|1-\bar{w}z|^{2+p\a}} dA(w)\right ) dA_{t}(z)
\end{displaymath}
with $\beta=\frac{(1+\varepsilon)(p-2)+\a p}{2}$. Note that the choice of $\varepsilon$ gives $p\a>\beta$, and therefore we can use Lemma \ref{Ict} again in order to obtain
\begin{displaymath}
||g-g(0)||^p_{X^p_\a}\le C \!\int_{\D} \!|g'(z)|^p\, (1-|z|^2)^{t+\beta-p\a} dA(z)=C\!\int_{\D}\! |g'(z)|^p\,(1-|z|^2)^{p-2}dA(z).
\end{displaymath}
\par Now assume that $1<p\le 2$.   Fix an $r$-lattice $\{a_ n\}$ with associated hyperbolic disks $\{D_ n\}$. Then
\begin{displaymath}
\begin{split}
||g-g(0)||^p_{X^p_\a}&\le \int_{\D}\left (\sum_ n \int_{D_ n} \frac{|g'(z)|^2\,dA_{\alpha}(z)}{|1-\bar{w}z|^{2+2\alpha}}\right )^{p/2}\!\!\!(1-|w|^2)^{p-2+\frac{\alpha p}{2}}\,dA(w)
\\
&\asymp \int_{\D}\!\!\left (\sum_ n \frac{(1-|a_ n|^2)^{\alpha}}{|1-\bar{a}_ n w|^{2+2\alpha}}\int_{D_ n} \!\!\!|g'(z)|^2 dA(z) \!\! \right )^{p/2}\!\!\!\!(1-|w|^2)^{p-2+\frac{\alpha p}{2}}\,dA(w)
\\
&
\le \int_{\D}\sum_ n \!\frac{(1-|a_ n|^2)^{p\alpha /2}}{|1-\bar{a}_ n w|^{p+p\alpha}}\left (\!\int_{D_ n} \!\!\!|g'(z)|^2 dA(z) \!\!\right )^{p/2}\!\!\!\!(1-|w|^2)^{p-2+\frac{\alpha p}{2}}\,dA(w).
\end{split}
\end{displaymath}
Now, passing the sum outside the integral and using Lemma \ref{Ict} we get
\begin{displaymath}
\begin{split}
||g-g(0)||^p_{X^p_\a}&\le \sum_ n (1-|a_ n|^2)^{p\alpha /2}\left (\int_{D_ n} \!\!\!|g'(z)|^2 dA(z) \!\right )^{p/2}\!\!\!\int_{\D}\!\frac{(1-|w|^2)^{p-2+\frac{\alpha p}{2}}\,dA(w)}{|1-\bar{a}_ n w|^{p+p\alpha}}
\\
& \le \sum_ n \left (\int_{D_ n} |g'(z)|^2\,dA(z)\right )^{p/2}<\infty,
\end{split}
\end{displaymath}
where the last step follows from Theorem 0 of \cite{AS} (see also \cite{ZhuJMAA91}).
This completes the proof.
\end{proof}

\subsection{Proof of Theorem \ref{th:nuevo}.}
\par The sufficiency for the case $1<p\le 2$, and the necessity for $2\le p<\infty$ is a byproduct of the following result, which also gives
some information on the case $p(1-\alpha)>4$.

\begin{prop}\label{Pr4} Let $g\in H(\D)$ and $\alpha\ge 0$.

\begin{enumerate}
\item[\rm(i)] If $1<p\le 2$  and $g\in \Xpa$,
then $T_g\in \mathcal{S}_p(\Da)$.
\item[\rm(ii)] If $2\le p<\infty$ and $T_g\in S_p(\Da)$ then  $g\in \Xpa$.
\end{enumerate}
\end{prop}

\begin{proof} Since
\begin{displaymath}
||g||^p_{X^p_\alpha}\asymp \int_{\D}\|T_g(j^{\alpha}_z)\|_{\Da}^p\,d\lambda (z),
\end{displaymath}
the result follows directly from Proposition \ref{pr:gen}.
\end{proof}
The necessity for $1<p<2$ follows from the next Proposition and part (iv) of Proposition \ref{pr:xpaproperties}.
\begin{prop}\label{nece}
Let $0\le \alpha<1$ and $g\in H(\D)$. Then
\begin{enumerate}
\item[(i)] If $1<p<2$ and $T_ g\in \mathcal{S}_ p(\Da)$, then $g\in B_ p$.

\item[(ii)] If $T_ g\in \mathcal{S}_ 1(\Da)$, then $g$ is constant.
\end{enumerate}
\end{prop}
\begin{proof}
Let $1\le p<2$, and assume that  $T_ g\in \mathcal{S}_ p(\Da)$. Then the positive operator $T_ g^*T_ g$
belongs to $\mathcal{S}_{p/2}(\Da)$. Without loss of generality we may
assume that $g'\neq 0$. Suppose
\[T^*_ g T_ g f=\sum_ n \lambda_ n \langle f,e_ n \rangle \,e_ n\]
is the canonical decomposition of $T_ g^*T_ g$. Then not only is
$\{e_ n\}$ an orthonormal set, it is also an orthonormal basis.
Indeed, if there is an unit vector $e\in \Da$ such that
$e\perp e_ n$ for all $n\geq 1$, then
\begin{displaymath}
\int_{\D} |g'(z)|^2 |e(z)|^2 \,dA_{\alpha}(z)= \|T_ g
e\|_{\Da}^2=\langle T_ g^*T_ g e,e\rangle_{\Da} =0
\end{displaymath}
because $T_ g^*T_ g$ is a linear combination of the vectors $e_ n$.
This would give $g'\equiv 0$.
\par Since $\{e_ n\}$ is an orthonormal basis of $\Da$,
then  by Lemma \ref{le:2}
\begin{equation}
\begin{split}\label{eq:2}
 \int_\D|g'(z)|^p &(1-|z|^2)^{p-2}\,dA(z)
 \\
 &\le C \int_\D|g'(z)|^p\left(\sum_n|e_n(z)|^p|e'_n(z)|^{2-p}\right)\,dA_{\alpha}(z)
\\
& \le C
\sum_n\left(\int_\D|g'(z)|^2|e_n(z)|^2\,dA_{\alpha}(z)\right)^{p/2}
\\
&=C \sum_ n \langle T_ g^*T_ g e_ n,e_ n \rangle_{\Da}
^{p/2} =C\sum_ n \lambda_ n^{p/2}=C\|T_ g^*T_
g\|_{\mathcal{S}_{p/2}}^{p/2},
\end{split}
\end{equation}
which finishes the proof of (i). Furthermore, if  $T_g\in \mathcal{S}_ 1(\Da)$, then
\eqref{eq:2} says that
$$  \int_\D|g'(z)| (1-|z|^2)^{-1}\,dA(z)<\infty,$$
which implies that $g$ is constant. This completes the proof.
\end{proof}

\par The remaining part of the proof is more involved. It  will be splitted in two cases.
\par {\textbf{Sufficiency. Case $\mathbf{2<p\le 4}$}.}
Let $\{e_ n\}$ be any orthonormal set in $\mathcal{D}_{\alpha}$. Then
\begin{displaymath}
\sum_ n \|T_ g e_ n\|_{\Da}^p=\sum_ n \left (\int_{\D} |g'(z)|^2 \,|e_ n (z)|^2 dA_{\alpha}(z)\right )^{p/2}\le C ( I_ 1+I_ 2),
\end{displaymath}
with
\begin{displaymath}
I_ 1=\sum_ n |e_ n(0)|^p\left (\int_{\D} |g'(z)|^2\,dA_{\alpha}(z)\right )^{p/2}
\end{displaymath}
and
\begin{displaymath}
I_ 2=\sum_ n \left (\int_{\D} |g'(z)|^2 \,|e_ n^2 (z)-e_ n^2(0)| dA_{\alpha}(z)\right )^{p/2}.
\end{displaymath}
Since $g\in \Xpa\subset \Da$ by Lemma \ref{pr:xpaproperties}  and $|e_ n(0)|\le 1$, we clearly have
\begin{displaymath}
I_ 1\le \|g\|_{\Da}^p \sum_ n |e_ n(0)|^2 \le \|g\|_{\Da}^p\,\|K^{\alpha}_ 0\|^2_{\Da}\le C\,\|g\|_{\Xpa}^p.
\end{displaymath}
In order to deal with the term $I_ 2$, note first that $e_n^2\in\mathcal{D}_{1+2\alpha}$ because for any $f\in\mathcal{D}_{\alpha}$,
$$|f(z)|^2\le C \frac{\|f\|^2_{\mathcal{D}_{\alpha}}}{(1-|z|)^\alpha},\quad z\in\D.$$
So  from the reproducing formula for $\mathcal{D}_{1+2\alpha}$ we deduce
\[|e_ n^2(z)-e_ n^2(0)|\le C \int_{\D} \frac{|e_n(w)||e'_ n(w)|}{|1-\bar{w}z|^{2+2\alpha}} dA_{1+2\alpha}(w).\]
Therefore, if we use the notation
\[J_ n(g):=\int_{\D} |g'(z)|^2 \,|e_ n^2 (z)-e_ n^2(0)| dA_{\alpha}(z),\]
 Fubini's theorem and H\"{o}lder's inequality yields
\begin{displaymath}
\begin{split}
J_ n(g)^{p/2}
& \le C \left(\int_{\D} |e_n(w)||e'_ n(w)|
\left[\int_{\D} \frac{|g'(z)|^2\,dA_{\alpha}(z)}{|1-\bar{w}z|^{2+2\alpha}}\right]\, dA_{1+2\alpha}(w) \right)^{p/2}
\\
&
\le C \int_{\D} |e_n(w)|^{\frac{p}{2}}\,|e'_ n(w)|^{2-\frac{p}{2}}
\left(\int_{\D} \frac{|g'(z)|^2\,dA_{\alpha}(z)}{|1-\bar{w}z|^{2+2\alpha}}\right)^{p/2}\! dA_{(1+\alpha)\frac{p}{2}+\alpha}(w).
\end{split}
\end{displaymath}
Then, if $p=4$, it follows from \eqref{eqRK1} and the fact that  $\|K^{\alpha}_ w\|^2_{\Da}\asymp(1-|w|^2)^{-\alpha}$  that
\begin{equation*}\begin{split}
I_ 2  = \sum_ n J_ n(g)^{2}
 \le & C \int_{\D}
 \|K^{\alpha}_ w\|^2_{\Da} \left(\int_{\D} \frac{|g'(z)|^2\,dA_{\alpha}(z)}{|1-\bar{w}z|^{2+2\alpha}}\right)^{2}\! dA_{(1+\alpha)2+\alpha}(w)
 \\ & \le C \int_{\D}
\left(\int_{\D} \frac{|g'(z)|^2\,dA_{\alpha}(z)}{|1-\bar{w}z|^{2+2\alpha}}\right)^{2}\! dA_{(1+\alpha)2}(w)=||g||^{4}_{X^4_\a}.
\end{split}\end{equation*}
Now, if $2<p<4$, notice that H\"{o}lder's inequality with exponent $4/p>1$ and \eqref{eqRK1} yield
\begin{displaymath}
\begin{split}
\sum_ n |e_n(w)|^{\frac{p}{2}}\,|e'_ n(w)|^{2-\frac{p}{2}}&\le \left (\sum_ n |e_ n(w)|^2 \right)^{p/4}\!\left (\sum _ n |e'_ n(w)|^2 \right )^{1-\frac{p}{4}}
\\
& \le \|K_ w^{\alpha}\|_{\Da}^{p/2} \,\,\|J_ w ^{\alpha}\|_{\Da}^{\frac{(4-p)}{2}}.
\end{split}
\end{displaymath}
This together with the fact that for $\alpha>0$ we have $\|K^{\alpha}_ w\|^2_{\Da}\asymp(1-|w|^2)^{-\alpha}$ and $\|J_ w^{\alpha}\|^2_{\Da}=(1-|w|^2)^{-(2+\alpha)}$, gives
\begin{displaymath}
\begin{split}
I_ 2 & =\sum_ n J_ n(g)^{p/2}
\\ &
\le C \!\!\int_{\D} \sum_ n |e_n(w)|^{\frac{p}{2}}\,|e'_ n(w)|^{2-\frac{p}{2}}
\left(\int_{\D} \frac{|g'(z)|^2\,dA_{\alpha}(z)}{|1-\bar{w}z|^{2+2\alpha}}\right)^{p/2}\! dA_{(1+\alpha)\frac{p}{2}+\alpha}(w)
\\
&
\le C \!\!\int_{\D} \|K_ w^{\alpha}\|_{\Da}^{p/2} \,\,\|J_ w ^{\alpha}\|_{\Da}^{\frac{4-p}{2}}
\left(\int_{\D} \frac{|g'(z)|^2\,dA_{\alpha}(z)}{|1-\bar{w}z|^{2+2\alpha}}\right)^{p/2}\! dA_{(1+\alpha)\frac{p}{2}+\alpha}(w)
\\
&
\le C\int_{\D}
\left(\int_{\D} \frac{|g'(z)|^2\,dA_{\alpha}(z)}{|1-\bar{w}z|^{2+2\alpha}}\right)^{p/2}\! \, dA_{p-2+\alpha\frac{p}{2}}(w)
=C\,||g||^p_{\Xpa}.
\end{split}
\end{displaymath}
 Since $g\in \Xpa$  combining the estimates for $I_ 2$ and $I_ 1$ we obtain that
\begin{displaymath}
\sum_ n \|T_ g e_ n\|_{\Da}^p\le C<\infty.
\end{displaymath}
Thus, by \cite[Theorem $1.33$]{Zhu}, the operator $T_g$ belongs to $\mathcal{S}_ p(\Da)$.
\medskip
\par {\textbf{Sufficiency. Case $\mathbf{4<p<\infty}$ and  $\mathbf{p(1-\alpha)<4}$ }.}
 Proceeding as before we get
\begin{displaymath}
\begin{split}
J_ n(g)^{p/2}
& \le C \left(\int_{\D} |e_n(w)||e'_ n(w)|
\left[\int_{\D} \frac{|g'(z)|^2\,dA_{\alpha}(z)}{|1-\bar{w}z|^{2+2\alpha}}\right]\, dA_{1+2\alpha}(w) \right)^{p/2}
\\
&
\le C \left(\int_{\D} |e_n(w)|^2 \,S_{\alpha}g(w)^2
\, dA_{2+\alpha}(w) \right)^{p/4}
\end{split}
\end{displaymath}
where
$$S_{\alpha}g(w)=(1-|w|^2)^{\alpha}\int_{\D} \frac{|g'(z)|^2\,dA_{\alpha}(z)}{|1-\bar{w}z|^{2+2\alpha}}.$$
Since $\a>0$, $p>4$ and $p(1-\alpha)<4$, H\"{o}lder's inequality implies that
$\|S_{\alpha}g\|^2_{L^2(\D,dA_{2+\alpha})}\le C \|g\|_{\Xpa}^4$, and therefore we can assume that $e_ n(0)=0$.
 Note  that for $\beta\ge \alpha$ we have
 \begin{equation}\label{eq:j10}
|e_n(w)|=|e_ n(w)-e_ n(0)|\le C \int_{\D} \frac{|e'_ n(\zeta)|}{|1-\bar{\zeta}w|^{1+\beta}} dA_{\beta}(\zeta).
\end{equation}
This follows from the reproducing formula for $\mathcal{D}_{\beta}$ and the fact that $\Da\subset \mathcal{D}_{\beta}$ if $\alpha\le \beta$.
 Since $p\alpha>p-4$, we can take $\varepsilon>0$ so that
 \begin{equation}\label{epsilon}
 \alpha p-3\varepsilon p>p-4.
 \end{equation}  Now, choose
  \begin{equation}\label{beta}
 \b>\max\left\{1+\frac{p\varepsilon}{2}, \frac{\alpha p-3\varepsilon p}{p-4},\a-\ep+\frac{4}{p} \right\}.
 \end{equation}
An application of Cauchy-Schwarz inequality and Lemma \ref{Ict}, together with \eqref{beta}  gives
\begin{equation}\label{Eqen}
|e_ n(w)|^2\le C \left (\int_{\D} \frac{|e'_ n(\zeta)|^2}{|1-\bar{\zeta}w|^{\beta}}\,dA_{\beta+\varepsilon}(\zeta) \right ) (1-|w|^2)^{-\varepsilon}.
\end{equation}
The use of \eqref{Eqen}, Fubini's theorem and H\"{o}lder's inequality give
\begin{displaymath}
\begin{split}
J_ n(g)^{p/2}& \le C \left(\int_{\D} |e'_ n(\zeta)|^2 \,\left(\int_{\D}\frac{S_{\alpha}g(w)^2}{|1-\bar{\zeta}w|^{\beta}}
\, dA_{2+\alpha-\varepsilon}(w)\right) dA_{\beta+\varepsilon}(\zeta)\right)^{p/4}
\\
&\le C \int_{\D} |e'_ n(\zeta)|^2 \,\left(\int_{\D}\frac{S_{\alpha}g(w)^2}{|1-\bar{\zeta}w|^{\beta}}
\, dA_{2+\alpha-\varepsilon}(w)\right)^{p/4} dA_{(\beta +\varepsilon)p/4+\alpha(1-\frac{p}{4})}(\zeta).
\end{split}
\end{displaymath}
Thus, by \eqref{eqRK1}
\begin{equation}\label{Eq6A}
\sum_ n J_ n(g)^{p/2}  \le C \!\!\int_{\D} \left(\int_{\D}\frac{S_{\alpha}g(w)^2}{|1-\bar{\zeta}w|^{\beta}}
\, dA_{2+\alpha-\varepsilon}(w)\right)^{p/4} dA_{-2+(\beta+\varepsilon-\alpha) \frac{p}{4}}(\zeta).
\end{equation}
Let $$\gamma= \frac{-2p+8+\alpha p-3\varepsilon p}{p-4}=-2+\frac{\alpha p-3\varepsilon p}{p-4}.$$
 By \eqref{epsilon}, we have $\gamma>-1$. Now, using H\"{o}lder's inequality, \eqref{beta} and Lemma \ref{Ict} we obtain
\begin{displaymath}
\begin{split}
\int_{\D}&\frac{S_{\alpha}g(w)^2}{|1-\bar{\zeta}w|^{\beta}}
\, dA_{2+\alpha-\varepsilon}(w)
\\
&
\le \left (\int_{\D}\frac{S_{\alpha}g(w)^{p/2}}{|1-\bar{\zeta}w|^{\beta}}
\, dA_{p-2+\varepsilon \frac{p}{2}}(w)\right)^{\frac{4}{p}} \left (\int_{\D} \frac{dA_{\gamma}(w)}{|1-\bar{\zeta}w|^{\beta}}\right)^{\frac{p-4}{p}}
\\
&
\le \left (\int_{\D}\frac{S_{\alpha}g(w)^{p/2}}{|1-\bar{\zeta}w|^{\beta}}
\, dA_{p-2+\varepsilon \frac{p}{2}}(w)\right)^{\frac{4}{p}} \left ((1-|\zeta|^2)^{\gamma+2-\beta}\right)^{\frac{p-4}{p}}.
\end{split}
\end{displaymath}
Putting this into \eqref{Eq6A}, Fubini's theorem an Lemma \ref{Ict} yields
\begin{displaymath}
\begin{split}
\sum_ n J_ n(g)^{p/2}  &\le C \!\!\int_{\D} \left (\int_{\D}\frac{S_{\alpha}g(w)^{p/2}}{|1-\bar{\zeta}w|^{\beta}}
\, dA_{p-2+\varepsilon \frac{p}{2}}(w)\right) (1-|\zeta|^2)^{\beta-2-\varepsilon \frac{p}{2}}dA(\zeta)
\\
&=C \int_{\D} S_{\alpha}g(w)^{p/2}\left (\int_{\D} \frac{ (1-|\zeta|^2)^{\beta-2-\varepsilon \frac{p}{2}}}{|1-\bar{\zeta}w|^{\beta}} dA(\zeta)\right ) dA_{p-2+\varepsilon \frac{p}{2}}(w)
\\
&\le C\int_{\D} S_{\alpha}g(w)^{p/2} dA_{p-2}(w)=C ||g||^p_{\Xpa}.
\end{split}
\end{displaymath}

\subsection{The open case}
In relation with the open case $p(1-\alpha)\ge 4$, we provide a result which can be proved following the lines of the proof of Theorem \ref{th:nuevo} (case $p>4$), and therefore the proof will be omitted.
\begin{prop}\label{suficpmayor2}
Let $0<\alpha$, $p\ge 2$  and $g\in H(\D)$.
 If $g\in X^p_{\alpha-\ep}$ for some  $\ep\in (0,\alpha)$, then
$T_ g\in \mathcal{S}_ p(\Da)$.
\end{prop}

\par Obviously, $X^p_{\alpha-\ep}\subsetneq  \Xpa$ if $(1-\a)p\ge 2$ (see Lemma \ref{le:normagkxpa} below), so  Proposition \ref{suficpmayor2} gives a sufficient but not necessary condition for $T_g\in \mathcal{S}_p(\Da)$, $(1-\a)p\ge 2$. However, if $\alpha>0$ and $1<p<\infty$, those techniques which will be developed in the proof of Lemma \ref{le:ga}, together with Lemma \ref{LI2}, imply that for any $\beta>0$,
$$\|T_{g_a}\|_{\mathcal{S}_p(\Da)}\asymp \|g_a\|_{B_ p}\asymp||g_a||_{X^p_\beta},$$
where $g_ a(z)=(1-\bar{a}z)^{-\gamma}$, $\gamma>0$. In particular,  the previous result gives the right growth  for this family of functions.

\section{Schatten classes of $T_g$ on the classical
Dirichlet space}\label{sclassical}
\subsection{Case $\mathbf{p\le 2}$}
\vspace{1em}\par
 \begin{proof}[Proof of Theorem \ref{th:2}]
 \par Since $\mathcal{S}_ p(\mathcal{D})\subset \mathcal{S}_ 1(\mathcal{D})$ for $0<p\le 1$, the result follows from part (ii) of Proposition \ref{nece}.
 \end{proof}
\par\begin{proof}[Proof of Theorem \ref{th:3}]
Part (a) follows from part (i) of Proposition \ref{nece}, and part (c) is deduced in Proposition \ref{Pr4}. In order to prove part (b), assume that $1<p<2$. Then, for all orthonormal sets $\{e_n\}$ of
$\mathcal{D}$, we have
\begin{displaymath}
\begin{split}
\sum_n\left|\langle T_ge_n, e_n\rangle_{\mathcal{D}}\right|^p & \le
\sum_n \left(\int_\D|g'(z)|\,|e_n(z)|\,|e'_n(z)|\,dA(z) \right)^p
\\
& \le \sum_n \left(\int_\D|g'(z)|^p\,|e_n(z)|^p\,|e'_n(z)|^{2-p}\,dA(z)\right)\left(\int_\D|e'_n(z)|^2\,dA(z)\right)^{p/p'}
\\
& \le  \int_\D|g'(z)|^p\left(\sum_n|e_n(z)|^p|e'_n(z)|^{2-p}\right)\,dA(z)
\\
& \le  \int_\D|g'(z)|^p\left(\sum_n|e_n(z)|^2\right)^{p/2}\left(\sum_n|e'_n(z)|^{2}\right)^{1-p/2}\,dA(z)
\\
& \le C \int_\D|g'(z)|^p \left(\log\frac{e}{1-|z|}\right)^{p/2}(1-|z|^2)^{p-2}\,dA(z).
\end{split}
\end{displaymath}
Thus, by \cite[Theorem 1.27]{Zhu}, we deduce that $T_ g\in \mathcal{S}_ p(\mathcal{D})$ with $\|T_ g\|_{\mathcal{S}_ p}\le C \|g\|_{B_ p \log^{p/2}}$.
\end{proof}

\subsection{Testing functions for Schatten classes}
\par Our next goal consists of proving that Theorem \ref{th:3} gives the correct behavior of $\|T_ g\|_{\mathcal{S}_ p}$, $1<p<2$, at least for some families of functions. For the beginning, we deal with monomials.
\begin{lem}\label{le:normagkxpa}
Asumme that $0\le \alpha<1$ and $1<p<\infty$. Let $g_j(z)=z^j$, $j=1,2,3\dots$. Then
 \begin{equation}\label{eq:normagkxpa}
\|T_{g_j}\|_{\mathcal{S}_ p(\Da)}\asymp ||g_j||_{X^p_\alpha}
\asymp
\begin{cases}
j^{\frac{1}{p}} & \textrm{if} \quad (1-\alpha)p<2,\\
\left(j\log(j+1)\right)^{\frac{1}{p}} & \textrm{if} \quad (1-\alpha)p=2,\\
j{\frac{1-\alpha}{2}} & \textrm{if} \quad (1-\alpha)p\ge 2,
\end{cases}
\end{equation}
\begin{equation}\label{2}
  \|g_j\|_{B_{p,\log^{p/2}}}\asymp j^{1/p}\left(\log (j+1)\right)^{1/2},
\end{equation}
and
\begin{equation}\label{j13}
\|g_ j\|_{B_p} \asymp j^{1/p}.
\end{equation}
\end{lem}
\begin{proof}
We shall use the inner product in $\Da$ given by
 \[\langle f,g\rangle=\sum_{k=0}^\infty(k+1)^{1-\alpha}\,a_k\,\overline{b_ k},\]
for $f(z)=\sum_{n=0}^\infty a_nz^n$,\, and $g(z)=\sum_{n=0}^\infty b_nz^n$.
We note that
\[T_{g_j}(f)(z)=j\,\sum_{k=0}^\infty \frac{a_k}{k+j}\,z^{k+j}=j\,\sum_{n=j}^\infty \frac{a_{n-j}}{n}\,z^{n}.\]
\par Now, if $\sigma_n(z)=\frac{z^n}{(n+1)^{\frac{1-\alpha}{2}}}$,\, $n\in\N$, we have that $\{\sigma_n\}_{n=0}^\infty$ is an orthonomal basis of $\Da$, and furthermore
\begin{displaymath}
\sum_{n=j}^\infty \frac{a_{n-j}}{n}\,z^{n}=\sum_{n=j}^\infty \frac{a_{n-j}(n+1)^{\frac{1-\alpha}{2}}}{n}\,\sigma_n=
\sum_{n=j}^\infty \frac{(n+1)^{\frac{1-\alpha}{2}}}{n(n-j+1)^{\frac{1-\alpha}{2}}}\langle
f,\sigma_{n-j}\rangle\sigma_n.
\end{displaymath}
 That is, the singular values of the integration operator\, $T_{g_j}$ are
$\left\{\frac{j(n+1)^{\frac{1-\alpha}{2}}}{n(n-j+1)^{\frac{1-\alpha}{2}}}\right\}_{n=j}^\infty$.
Consequently,
\begin{equation}\begin{split}\label{eq:j1}
\|T_{g_j}\|^p_{\mathcal{S}_p}\asymp j^p\sum_{n=j}^\infty
\frac{1}{\big(n^{1+\alpha}(n-j+1)^{1-\alpha}\big)^{p/2}}.
\end{split}
\end{equation}
\par On the other hand,
\begin{equation*}\begin{split}
 \|g_ j\|^p_{\Xpa}&= j^p\int_{\D}\left(\int_\D \frac{|\z|^{2(j-1)}}{|1-\overline{z}\z|^{2+2\alpha}}\,dA_\alpha(\zeta)\right)^{\frac{p}{2}} \,(1-|z|^2)^{p-2+\frac{p\alpha}{2}}\,dA(z).
\\
 & \asymp j^p\int_{\D}\left( \int_0^1 \frac{r^{2j-1}(1-r)^{\alpha}}{(1-r|z|)^{1+2\alpha}}\,dr\right)^{\frac{p}{2}} \,(1-|z|^2)^{p-2+\frac{p\alpha}{2}}\,dA(z)
\\
 & \asymp j^p \int_{\D}\left( \sum_{m=0}^\infty \frac{(m+1)^{2\alpha}|z|^{m}}{(2j+m)^{1+\alpha}} \right)^{\frac{p}{2}} \,(1-|z|^2)^{p-2+\frac{p\alpha}{2}}\,dA(z)
\\
& \asymp j^p \int_0^1 \left( \sum_{m=0}^\infty \frac{(m+1)^{2\alpha}s^{m}}{(2j+m)^{1+\alpha}} \right)^{\frac{p}{2}} \,(1-s^2)^{p-2+\frac{p\alpha}{2}}\,ds.
\end{split}
\end{equation*}
At this point, we use \cite[Theorem $1$]{MP} to obtain
\begin{equation*}
\begin{split}
\|g_ j\|^p_{\Xpa}& \asymp j^p \sum_{n=0}^\infty \frac{1}{2^{n(p-1+\frac{p\alpha}{2})}}\left(\sum_{m\in I(n)}\frac{(m+1)^{2\alpha}}{(2j+m)^{1+\alpha}}\right)^{p/2}
\\
& \asymp j^p \sum_{n=0}^\infty \frac{2^{n}}{\left[2^{n(1-\alpha)}(2^n+2j)^{1+\alpha}\right]^{p/2}}
\\
 & \asymp j^p \sum_{n=0}^\infty \left(\sum_{m\in I(n)}\frac{1}{\left[(m+1)^{(1-\alpha)}(m+j)^{(1+\alpha)}\right]^{p/2}}\right)
\\
& \asymp j^p \sum_{n=0}^\infty\frac{1}{\left[(n+1)^{(1-\alpha)}(n+j)^{(1+\alpha)}\right]^{p/2}}
\\
& \asymp j^p \sum_{n=j}^\infty\frac{1}{\left[n^{(1+\alpha)}(n-j+1)^{(1-\alpha)}\right]^{p/2}},
\end{split}\end{equation*}
which together with \eqref{eq:j1} gives the first equivalence in \eqref{eq:normagkxpa}. The second   equivalence in \eqref{eq:normagkxpa} follows from an straightforward calculation according to those values of $p$ and $\alpha$.
\par Now we prove \eqref{2},
\begin{equation*}\begin{split}\label{1}
\|g_j\|^p_{B_{p,\log^{p/2}}}&=j^p\int_0^1 r^{(j-1)p+1}(1-r)^{p-2}\left(\log\frac{e}{1-r}\right)^{p/2}\,dr
\\
& \ge j^p\int_{1-\frac{1}{j+1}}^1 r^{(j-1)p+1}(1-r)^{p-2}\left(\log\frac{e}{1-r}\right)^{p/2}\,dr
\\
 & \asymp j^p \int_{1-\frac{1}{j+1}}^1 (1-r)^{p-2}\left(\log\frac{e}{1-r}\right)^{p/2}\,dr
\\
& \asymp  j\left(\log (j+1)\right)^{p/2},
\end{split}\end{equation*}
where in the last step we have used that $\om(r)=(1-r)^{p-2}\left(\log\frac{e}{1-r}\right)^{p/2}$ is an admissible weight with distortion function equivalent to
$(1-r)$ (see \cite[p. 11]{PavP}).
Now, bearing in mind the properties of the Beta function,
\begin{equation*}\begin{split}
&j^p\int_0^{1-\frac{1}{j+1}} r^{(j-1)p+1}(1-r)^{p-2}\left(\log\frac{e}{1-r}\right)^{p/2}\,dr
\\
& \le C j^p \left(\log (j+1)\right)^{p/2}\int_0^{1-\frac{1}{j+1}} r^{(j-1)p+1}(1-r)^{p-2}\,dr
\\
& \le C j^p \left(\log (j+1)\right)^{p/2}\int_0^{1} r^{(j-1)p+1}(1-r)^{p-2}\,dr
\\
& \asymp  j\left(\log (j+1)\right)^{p/2},
\end{split}\end{equation*}
so we get \eqref{2}.
 The equivalence \eqref{j13} can be proved analogously. This finishes the proof.

\end{proof}
\medskip
\par Next, for each $a\in \D$ and $\gamma>0$, consider the functions $g_ a(z)=(1-\bar{a}z)^{-\gamma}$.
\begin{lem}\label{le:ga}
Assume that $p>1$ and $\gamma>0$. Then
\begin{equation}\label{eq:ga3}
\|g_ a\|_{B_{p}} \asymp (1-|a|^2)^{-\gamma},
\end{equation}
\begin{equation}\label{eq:ga1}
\|g_a\|_{X^p_{0}}\asymp (1-|a|)^{-\gamma}\left(\log\frac{e}{1-|a|}\right)^{1/p},\quad |a|\to 1^-,
\end{equation}
and
\begin{equation}\label{eq:ga2}
\|T_{g_ a}\|_{\mathcal{S}_p(\mathcal{D})}\asymp \|g_ a\|_{B_{p,\log^{p/2}}} \asymp (1-|a|^2)^{-\gamma}\,\left(\log \frac{e}{1-|a|^2}\right)^{1/2}.
\end{equation}
\end{lem}
\begin{proof}
A simple use of Lemma \ref{Ict} implies \eqref{eq:ga3}.
 Since $|1-\overline{w}a|\asymp |1-\overline{w}z|$ for any $z\in D(a)=\big\{z: |z-a|<\frac{1-|a|}{2}\big\}$, it follows that
\begin{equation*}\begin{split}
\int_\D\frac{dA(z)}{|1-\overline{w}z|^2|1-\overline{a}z|^{2\gamma+2}}& \ge
\int_{D(a)}\frac{dA(z)}{|1-\overline{w}z|^2|1-\overline{a}z|^{2\gamma+2}}
\\
& \ge \frac{C}{(1-|a|)^{2\gamma+2}}\int_{D_a}\frac{dA(z)}{|1-\overline{w}z|^2}
\\
 & \ge \frac{C}{(1-|a|)^{2\gamma}|1-\overline{w}a|^2}.
\end{split}\end{equation*}
Therefore, joining this and Lemma \ref{Ict},
\begin{displaymath}
\begin{split}
\|g_a\|^p_{X^p_{0}}= &|a\gamma|^p\int_\D\left(\int_\D\frac{dA(z)}{|1-\overline{w}z|^2|1-\overline{a}z|^{2\gamma+2}}\right)^{p/2} (1-|w|^2)^{p-2}\,dA(w)
\\
& \ge C  \frac{C}{(1-|a|)^{p\gamma}}\int_\D\frac{ (1-|w|^2)^{p-2}\,dA(w)}{|1-\overline{w}a|^p}
\\
& \asymp (1-|a|)^{-p\gamma}\left(\log\frac{e}{1-|a|}\right).
\end{split}
\end{displaymath}
\par On the other hand, taking $0<\varepsilon<\min(1,2(p-1)/p)$, and bearing in mind Lemma \ref{LI2},
\begin{displaymath}
\begin{split}
 \int_\D\frac{dA(z)}{|1-\overline{w}z|^2|1-\overline{a}z|^{2\gamma+2}} &\le (1-|w|)^{-\varepsilon}\int_\D\frac{dA(z)}{|1-\overline{w}z|^{2-\varepsilon}|1-\overline{a}z|^{2\gamma+2}}
 \\
& \le C \frac{(1-|w|^2)^{-\varepsilon}}{(1-|a|)^{2\gamma}|1-\overline{w}a|^{2-\varepsilon}}.
\end{split}
\end{displaymath}
So, an application of Lemma \ref{Ict} gives
\begin{equation*}
\begin{split}
\|g_a\|^p_{X^p_{0}} & \le \frac{C}{(1-|a|)^{p\gamma}}\int_\D\frac{ (1-|w|^2)^{p-2-\varepsilon p/2}\,dA(w)}{|1-\overline{w}a|^{p-\varepsilon \frac{p}{2}}}
\\
& \asymp (1-|a|)^{-p\gamma}\left(\log\frac{e}{1-|a|}\right),\quad |a|\to 1^-.
\end{split}
\end{equation*}
\par In order to prove \eqref{eq:ga2}, we first estimate the $B_ {p, \log^{p/2}}$-norm of the functions $g_ a(z)=(1-\bar{a}z)^{-\gamma}$.
Take $a\in\D$ with $|a|\ge 1/2$.
\begin{equation}
\begin{split}\label{gabplog}
\|g_a\|^p_{B_{p,\log^{p/2}}} & \asymp \int_0^1 \frac{(1-s)^{p-2}\left(\log\frac{e}{1-s}\right)^{p/2}\,ds}{(1-|a|s)^{p\gamma+p-1}}
\\
&=  \int_0^{|a|} \frac{(1-s)^{p-2}\left(\log\frac{e}{1-s}\right)^{p/2} ds}{(1-|a|s)^{p\gamma+p-1}}
 +
 \int_{|a|}^1 \frac{(1-s)^{p-2}\left(\log\frac{e}{1-s}\right)^{p/2} ds}{(1-|a|s)^{p\gamma+p-1}}
\\
 & = I_1(|a|)+I_2(|a|).
\end{split}
\end{equation}
\par Since $(1-s)^{p-2}\left(\log\frac{e}{1-s}\right)^{p/2}$ is an admissible weight
\begin{equation}
\begin{split}\label{gabplog2}
I_2(|a|)&\asymp \frac{1}{ (1-|a|)^{p\gamma+p-1}}\int_{|a|}^1(1-s)^{p-2}\left(\log\frac{e}{1-s}\right)^{p/2}\,ds
\\
& \asymp  \frac{1}{ (1-|a|)^{p\gamma}}\left(\log\frac{e}{1-|a|}\right)^{p/2}\quad |a|\to 1^-.
\end{split}
\end{equation}
Moreover,
\begin{equation*}
\begin{split}
I_1(|a|)\le & \left(\log\frac{e}{1-|a|}\right)^{p/2}\int_0^{|a|} \frac{(1-s)^{p-2}\,ds}{(1-|a|s)^{p\gamma+p-1}}
\\
& \le C \left(\log\frac{e}{1-|a|}\right)^{p/2}\frac{1}{ (1-|a|)^{p\gamma}}.
\end{split}
\end{equation*}
which together with \eqref{gabplog} and \eqref{gabplog2} gives
\begin{equation}
\begin{split}\label{gabplog3}
\|g_a\|_{B_{p,\log^{p/2}}}\asymp (1-|a|)^{-\gamma}\left(\log\frac{e}{1-|a|}\right)^{1/2},\quad |a|\to 1^-.
\end{split}
\end{equation}
\par Now, if $1<p\le 2$, by \eqref{gabplog3}, the description of Hilbert-Schmidt integration operators obtained in \eqref{HS}, and  part (b) of Theorem \ref{th:3},
\begin{equation}
\begin{split}\label{gasp}
&(1-|a|)^{-\gamma}\left(\log\frac{e}{1-|a|}\right)^{1/2}  \asymp \|g_a\|_{B_{2,\log^{1}}}
\asymp \|T_{g_a}\|_{\mathcal{S}_2(\mathcal{D})}  \le \|T_{g_a}\|_{\mathcal{S}_p(\mathcal{D})}
\\ & \le C\|g_a\|_{B_{p,\log^{p/2}}}
 \asymp (1-|a|)^{-\gamma}\left(\log\frac{e}{1-|a|}\right)^{1/2},\quad |a|\to 1^-.
\end{split}
\end{equation}
Furthermore, if $2\le p<\infty$, using again \eqref{gabplog3} and  Proposition \ref{pr:clasicop>2} below,
\begin{equation*}
\begin{split}
&(1-|a|)^{-\gamma}\left(\log\frac{e}{1-|a|}\right)^{1/2}  \asymp \|g_a\|_{B_{p,\log^{p/2}}}
\le C \|T_{g_a}\|_{\mathcal{S}_p(\mathcal{D})}  \le C \|T_{g_a}\|_{\mathcal{S}_2(\mathcal{D})}
\\
& \asymp \|g_a\|_{B_{2,\log^{1}}}
\asymp (1-|a|)^{-\gamma}\left(\log\frac{e}{1-|a|}\right)^{1/2},\quad |a|\to 1^-,
\end{split}
\end{equation*}
and this completes the proof of (b).
\end{proof}
Bearing in mind that $\left(\Xpa,||\cdot||_{\Xpa}\right)$ is a Banach space for $p>1$, the closed graph theorem and Lemma \ref{le:normagkxpa} and Lemma \ref{le:ga}, we deduce that $X^p_0\subsetneq B_p$ and is different from $B_{p,\log^{p/2}}$. In particular, Proposition \ref{pr:xpaproperties} (iv) does not remain true for $\a=0$ and $1<p<2$.

\subsection{Case $\mathbf{p>2}$}
We collect our results for this range of values of $p$ in the next proposition.
\begin{prop}\label{pr:clasicop>2}
Assume that $g\in H(\D)$ and $2\le p<\infty$.
\begin{enumerate}
\item[\rm(i)]\,If  $T_ g\in \mathcal{S}_ p(\mathcal{D})$ then $g\in B_{p,\log^{p/2}}$.
\item[\rm(ii)]\,If  $T_ g\in \mathcal{S}_ p(\mathcal{D})$ then $g\in X^p_0$.
\item[\rm(iii)]\, Assume that $2< p\le 4$.
 If $$||g||^p_{X^p_{0,\log^{p/4}}}\!\!\ig |g(0)|^p+\int_{\D}\!\left (\int_{\D}\!\frac{|g'(z)|^2}{|1-\bar{w}z|^{2}}
dA(z)\!\right)^{p/2}\!\!\left(\!\log\frac{e}{1-|w|}\!\right)^{p/4}\!\!\!dA_{p-2}(w)<\infty,$$
then
$T_ g\in \mathcal{S}_ p(\mathcal{D})$.
\end{enumerate}
\end{prop}

\begin{proof}
\par
We prove part (i) first. It is clear that $T_ g \in \mathcal{S}_ p(\mathcal{D})$ if and
only if $M_{g'}\in \mathcal{S}_ p(\mathcal{D},A^2)$, which is equivalent to
the fact that the adjoint $M^*_{g'}$ belongs to $\mathcal{S}_
p(A^2,\mathcal{D})$. Now, the result can be deduced by applying Proposition \ref{th:general} with $T=M^*_{g'}$ and $H=\mathcal{D}$. Indeed, an easy computation using
\eqref{RKDA},  the properties of the adjoint and the reproducing kernels gives
\begin{displaymath}
\begin{split}
\|M_{g'}^* B^{0}_ z\|_{\mathcal{D}}^2&=\langle M_{g'}^* B^{0}_ z,
M_{g'}^* B^{0}_ z \rangle _{\mathcal{D}}=\langle B^{0}_z, M_{g'}
M_{g'}^* B^{0}_ z \rangle_{A ^2}
=\overline{M_{g'} M_{g'}^* B^{0}_ z (z)}
\\
&=\overline{g'(z)\, M_{g'}^* B^{0}_ z (z)}
 =\overline{g'(z)\, \langle M_{g'}^* B^{0}_ z, K^{\mathcal{D}}_
z\rangle _{\mathcal{D}} }
= \overline{g'(z)\, \langle
B^{0}_ z, M_{g'} K^{\mathcal{D}}_ z\rangle _{A^2}}
\\
&=\overline{g'(z)}\,\, M_{g'} K^{\mathcal{D}}_
z(z)=|g'(z)|^2\,\log \frac{e}{1-|z|^2}.
\end{split}
\end{displaymath}
This, bearing in mind \eqref{RKBergman}, yields
\begin{displaymath}
\begin{split}
\int_{\D}
|g'(z)|^p\,\left (\log \frac{e}{1-|z|^2}\right)^{p/2} \!(1-|z|^2)^{p-2}\,dA(z)&\asymp \int_{\D}\|M_{g'}^* b_ z^{0}\|_{\mathcal{D}}^p\,d\lambda (z),
\end{split}
\end{displaymath}
which together with Proposition \ref{th:general}, gives (i). Part (ii) follows from Proposition \ref{Pr4}. Finally, reasoning as in the proof of Theorem \ref{th:nuevo} (case $2<p\le 4$), we obtain
part (iii).
\end{proof}
\par By arguing
now similarly as
 in the proof of Lemma \ref{le:ga}, we deduce that
\begin{displaymath}
||g_ a||_{X^p_{0,\log^{p/4}}}\asymp (1-|a|)^{-\gamma}\left(\log\frac{1}{1-|a|}\right)^{\frac{1}{4}+\frac{1}{p}},\quad |a|\to 1^-,
\end{displaymath}
 and Proposition \ref{pr:clasicop>2}, together with  Lemma \ref{le:normagkxpa} and Lemma \ref{le:ga}, says that any of those conditions which appear in
Proposition \ref{pr:clasicop>2} does not describe the membership of $T_g$ in $S_p(\mathcal{D})$ for $p>2$. However,
 if the monomials are taken as the symbols, Lemma \ref{le:normagkxpa} says that the correct behavior of $\|T_ g\|_{\mathcal{S}_ p(\mathcal{D})}$ is given by $X^p_0$, but if the symbols are the family of functions  $g_ a(z)=(1-\bar{a}z)^{-\gamma}$, $a\in\D$, Lemma \ref{le:ga} says that the correct behavior is   given by the $B_{p,\log^{p/2}}$ condition.

\section{Relationship with other operators}\label{operators}
It should be noticed that the integration operator $T_ g$ is bounded, compact (in $\Da$), or belongs to the Schatten class $\mathcal{S}_ p(\Da)$ if and only if the multiplication operator $M_{g'}:\Da\rightarrow A^2_{\alpha}$ is bounded, compact, or belongs to $\mathcal{S}_ p$. In this section, we shall study the relationship of the integration operator $T_ g$ (equivalently $M_{g'}$) with other linear operators such as Toeplitz operators, the big and small Hankel operators, or other multiplication operators.
\subsection{Toeplitz operators}
We recall that given a finite positive Borel measure $\mu$ on $\D$, the Toeplitz operator $Q_{\mu}$ on $\Da$, $\alpha>0$ is defined by
\[Q_{\mu} f(z)=\int_{\D}f(w)\,\overline{K^{\alpha}_ z(w)}\,d\mu(w),\qquad f\in \Da.\]
  Toeplitz operators  have been a key tool for studying the membership in $\mathcal{S}_p$  of many classes of operators, such as composition operators (see \cite{LuZhu92}, \cite[Section $7$]{LJFA87} and \cite[Chapter 11]{Zhu}) or integration operators (see \cite{AS0, AS} and \cite[Chapter $6$]{PelRat}). Indeed,
  the integration operator $T_ g$ and the Toeplitz operator $Q_{\mu}$ on $\Da$ are related via the identity $T^*_ g T_ g=Q_{\mu_ g}$, where  $\mu_ g$ is the measure defined by $d\mu_ g(z)=|g'(z)|^2\,dA_{\alpha}(z)$, and one can obtain a proof of Theorem \ref{th:1} using the characterization of Schatten class Toeplitz operators obtained by D. Luecking (see \eqref{LC} below). So, it is natural to expect that the methods used to study the membership of $T_{g}$ in the Schatten $p$-class of $\Da$ are going to work also for the Toeplitz operator $Q_{\mu}$ on $\Da$ for a general measure $\mu$. Before doing that,  we recall Luecking's result \cite{LJFA87} describing the membership in $\mathcal{S}_ p(\Da)$ of the Toeplitz operator $Q_{\mu}$ for all $p>0$ with $p(1-\alpha)<1$. He shows that, for the range of $p$ considered above, $Q_{\mu}\in \mathcal{S}_ p(\Da)$ if and only if, for any $r$-lattice $\{a_ j\}$ with associated hyperbolic disks $\{D_ j\}$
\begin{equation}\label{LC}
\sum_ j \left (\frac{\mu(D_ j)}{(1-|a_ j|)^{\alpha}}\right )^p<\infty.
\end{equation}
Given a finite positive Borel measure on $\D$,  for any $-1< \a<\infty$ and $0<p<\infty$  we define
 \begin{equation*}\label{eq:xpmu}
X^p_{\alpha}(\mu) \ig \int_{\D}\!\left ((1-|w|^2)^\alpha\int_{\D}\frac{d\mu(z)}{|1-\bar{w}z|^{2+2\a}}\!\right)^{p/2}\!\!\!\!(1-|w|^2)^{p-2}dA(w).
\end{equation*}
 Here we are able to obtain a full description of the measures $\mu$ for which the Toeplitz operator $Q_{\mu}$ belongs to $\mathcal{S}_ p(\Da)$  on  the extended range of all $p>0$ with $p(1-\alpha)<2$ and  $1<p(2+\a)$. We remark here that, as $\alpha>0$,
  %this is a true restriction only if $p>2$, and therefore
    a complete description of the Hilbert-Schmidt Toeplitz operators on $\Da$ is obtained.
\begin{thm}\label{th:TC}
Let $\mu$ be a finite positive Borel measure on $\D$, $\alpha>0$, and let $p>0$ with $1<p(2+\a)$ and  $p(1-\alpha)<2$. Then the Toeplitz operator $Q_{\mu}$ belongs to $\mathcal{S}_ p(\Da)$ if and only if $X_{\alpha}^{2p}(\mu)<\infty$.
%\begin{equation}\label{TC}
%X_{\alpha}^{2p}(\mu)=\int_{\D}\!\left ((1-|z|^2)^{\alpha}\!\!\int_{\D}\frac{d\mu(w)}{|1-\bar{w}z|^{2+2\alpha}} \right )^p \! (1-|z|^2)^{2p-2}dA(z)<\infty.
%\end{equation}
\end{thm}
\begin{proof}
Consider the inclusion operator $I_{\mu}:\Da\rightarrow L^2(\D,\mu)$. It is easy to check that $Q_{\mu}=I_{\mu}^* I_{\mu}$, and thus $Q_{\mu} \in \mathcal{S}_ p(\Da)$ if and only if $I_{\mu}$ belongs to $\mathcal{S}_{2p}$. Now, the necessity of $X_{\alpha}^{2p}(\mu)<\infty$ for $p\ge 1$ and the sufficiency for $p\le 1$ follow from  Proposition \ref{pr:gen}. Also, by repeating the proof of the sufficiency in Theorem \ref {th:nuevo} replacing the measure $|g'(z)|^2 \,dA_{\alpha}(z)$ in that proof by the measure $d\mu$ we obtain
\[\sum_ n \|I_{\mu} e_ n \|_{L^2(\D,\mu)}^{2p}\le C<\infty\]
for all orthonormal sets $\{e_ n\}$ of $\Da$ provided $p> 1$ and $p(1-\alpha)<2$. This proves
the sufficiency of $X_{\alpha}^{2p}(\mu)<\infty$ in that range. Finally, it remains to show the necessity in the case $1/(2+\alpha)<p<1$.  Let $\{a_ j\}$ be an $r$-lattice with associated hyperbolic disks $\{D_ j\}$. Using that $|1-\bar{w}z|\asymp |1-\bar{a}_ j z|$ for $w\in D_ j$ and Lemma \ref{Ict}, we deduce
\begin{displaymath}
\begin{split}
X_{\alpha}^{2p}(\mu) &\le C \int_{\D} \left (\sum_ j \frac{\mu(D_ j)}{|1-\bar{a}_ j z|^{2+2\alpha}}\right )^p (1-|z|^2)^{2p-2+\alpha p}\, dA(z)
\\
& \le C \sum_ j \mu (D_ j)^p \int_{\D} \frac{(1-|z|^2)^{2p-2+\alpha p}}{|1-\bar{a}_ j z|^{2p+2\alpha p}}\,dA(z)
\\
& \le C \sum_ j \frac{\mu (D_ j)^p }{(1-|a_ j|^2)^{\alpha p}}.
\end{split}
\end{displaymath}
Thus, by Luecking's condition \eqref{LC}, if $Q_{\mu}\in \mathcal{S}_ p(\Da)$ then $X_{\alpha}^{2p}(\mu)<\infty$ completing the proof of the Theorem.
\end{proof}
We conclude this subsection mentioning that in \cite{RW1} one can find a description of the membership of the Toeplitz operator $Q_{\mu}$ in $\mathcal{S}_{2k}(\Da)$ for positive integers $k$ in terms of some iterated integrals.
\subsection{Big and small Hankel operators}
As in
\cite{WuIE92} and \cite{RW93}, for $\alpha\ge 0$, we consider the Sobolev space $L^2_{\alpha}$ consisting of those differentiable functions $u:\D\rightarrow \mathbb{C}$ for which the norm
 \[\|u\|_{L^2_{\alpha}}=\left (|u(0)|^2+\int_{\D} |\nabla u(z)|^2\,dA_{\alpha}(z)\right )^{1/2}\]
  is finite. It is clear that $\Da$ is a closed subspace of $L^2_{\alpha}$. Let $P_{\alpha}$ be the orthogonal projection from  $L^2_{\alpha}$ onto $\Da$. The big Hankel operator $H_ g^{\alpha}:\Da \rightarrow L^2_{\alpha}$ and the small Hankel operator $h_ g^{\alpha}:\Da \rightarrow L^2_{\alpha}$ are defined by
\begin{equation}\label{DefHO}
\begin{split}
&H^{\alpha}_ g (f)=(I-P_\alpha)(\overline{g}f),
\\
&h^{\alpha}_g(f)=\overline{P_\alpha(\overline{f}g)}.
\end{split}
\end{equation}
The relation between the big Hankel operator and the multiplication operator $M_{g'}$ is clear and well understood. Indeed,
in \cite[Corollary 1]{WuIE92} Z. Wu shows that $M_{g'}:\Da\rightarrow A^2_{\alpha}$ is bounded, compact, or belongs to $\mathcal{S}_ p$ with $1<p<\infty$, if and only if the same is true for the big Hankel operator $H_{g}^{\alpha}:\Da \rightarrow L^2_{\alpha}$.
 However, although $M_{\bar{g'}}$ is related with the the small Hankel operator (see \eqref{eq:relationmghg} below), the transformation of a result from one  operator to the other is not straightforward. Respect to this question,  it is known
  that  $M_{g'}^{\alpha}:\Da \rightarrow A^2_{\alpha}$ is bounded (or compact) if and only if $h_ g^{\alpha}:\Da \rightarrow L^2_{\alpha}$ is bounded (or compact) (see Theorem 2 and Lemma 3.3 of \cite{Wuark93}). Z. Wu also shows that (in the case of the Dirichlet space) for $p\ge 2$, $H_{g}^{0}:\mathcal{D} \rightarrow L^2_{0}$ belongs to $\mathcal{S}_ p$ if and only if $h_ g^{0}:\mathcal{D} \rightarrow L^2_{0}$ belongs to $\mathcal{S}_ p$ (see \cite[Theorem 6]{WuIE92}. Note that, by the previous observations, we may replace $H_{g}^{0}$ by $M_{g'}$ or $T_ g$). The main aim of this section consists of extending Wu's result on Schatten $p$-classes for the small Hankel operator to all $\Da$ and to all $p$ with $1<p<\infty$. Before that, we recall that
\[P_{\alpha}u(w)=u(0)+\int_{\D}\frac{\partial u}{\partial z}(z)\,\overline{\frac{\partial K_ w^{\alpha}(z)}{\partial z}}\,dA_{\alpha}(z), \]
and has the property (see \cite[p.105]{RW93}) that
\begin{equation}\label{EqS7}
\frac{\partial}{\partial w }\big ( P_{\alpha} u\big )(w)=\int_{\D} \frac{\partial u}{\partial z}(z) \frac{dA_{\alpha}(z)}{(1-\bar{z}w)^{2+\alpha}},\qquad u\in L^2_{\alpha}.
\end{equation}
\begin{thm}\label{th:small Hankel}
Let $\alpha\ge 0$, $g\in H(\D)$ and $1<p<\infty$. Then $T_ g\in \mathcal{S}_ p(\Da)$ if and only if $h_ g^{\alpha}\in \mathcal{S}_ p(\Da, L^2_{\alpha})$.
\end{thm}
\begin{proof}
Firstly, we recall  that  if $T_ g$ or $h_ g^{\alpha}$ is bounded, then $g\in \Da$. It is enough to consider  the relationship between $M_{\bar{g'}}$ and $h^{\alpha}_ g$. For this, we look at the difference of $M_{\bar{g'}}$ and $\frac{\partial}{\partial \overline{w}}h_ g^{\alpha}$.
For $f\in \Da$, a straightforward calculation using that $g\in \Da$ and \eqref{EqS7} yields
\begin{equation}\label{eq:relationmghg}
M_{\bar{g'}}f(w)-\frac{\partial}{\partial \overline{w}}(h_ g^{\alpha}\!f)(w)=\overline{\int_{\D}g'(z)\frac{\overline{f(w)}-\overline{f(z)}}{(1-\bar{z}w)^{2+\alpha}}\,dA_{\alpha}(z)}.
\end{equation}
For $1<p<\infty$, if $T_ g \in \mathcal{S}_ p(\Da)$ or $h_ g^{\alpha}\in \mathcal{S}_ p(\Da, L^2_{\alpha})$ then $g\in B_ p$ (see Propositions \ref{pr:xpaproperties}, \ref{Pr4}, \ref{nece}, Theorem \ref{th:3} and \cite[Theorem 1]{WuIE92}), and therefore the difference considered above, as an operator acting from $\Da$ into $L^2(\D,dA_{\alpha})$, belongs to $\mathcal{S}_ p$, by Proposition \ref{PH1} (which we are going to prove below). This completes the proof.
\end{proof}
For $u\in L^2(\D,dA_{\alpha})$, consider the operator
\[\Delta_ u f(w)=\int_{\D}u(z)\frac{\overline{f(w)}-\overline{f(z)}}{(1-\bar{z}w)^{2+\alpha}}\,dA_{\alpha}(z).\]

\begin{prop}\label{PH1}
Let $\alpha\ge 0$, $u\in A^2_{\alpha}$ and $1<p<\infty$.
 If $u\in A^p_{p-2}$, then $\Delta_ u:\Da\rightarrow L^2(\D,dA_{\alpha})$ belongs to $\mathcal{S}_ p$.
\end{prop}
For the proof of that proposition, we need the following lemma.
\begin{lem}\label{LS1}
Let  $\sigma>-1$, and $2+\sigma<b\leq 4+2\sigma$.
Then for each $a\in \D$ and any $f\in H(\D)$ we have
\begin{displaymath}
\int_{\D}\frac{|f(z)-f(a)|^2}{|1-\bar{a}z|^{b}}\,dA_{\sigma}(z)\le C
\int_{\D} |f'(z)|^2\,\frac{dA_{2+\sigma}(z)}{|1-\bar{a}z|^{b}}.
\end{displaymath}
\end{lem}
\begin{proof}
Let $\varphi_ a(z)=\frac{a-z}{1-\bar{a}z}$, and consider the function $f_ a=(f \circ \,\varphi_ a)$. After the change of variables
$z=\varphi_ a(\zeta)$, and an application of Lemma $2.1$ of \cite{BP} we
get
\begin{displaymath}
\begin{split}
\int_{\D}\frac{|f(z)-f(a)|^2}{|1-\bar{a}z|^{b}}\,dA_{\sigma}(z)&=(1-|a|^2)^{2+\sigma-b}\int_{\D}\frac{|f_
a(\zeta)-f_
a(0)|^2}{|1-\bar{a}\zeta|^{4+2\sigma-b}}\,dA_{\sigma}(\zeta)
\\
&\le C (1-|a|^2)^{2+\sigma-b}\int_{\D}\frac{|(f_
a)'(\zeta)|^2\,dA_{2+\sigma}(\zeta)}{|1-\bar{a}\zeta|^{4+2\sigma-b}}.
\end{split}
\end{displaymath}
Finally, the change of variables $\zeta=\varphi_ a(z)$ gives
\begin{displaymath}
\begin{split}
\int_{\D}\frac{|f(z)-f(a)|^2}{|1-\bar{a}z|^{b}}\,dA_{\sigma}(z)
&\le C \!\!\int_{\D}|f'(z)|^2
\frac{|1-\bar{a}z|^{4+2\sigma-b}}{(1-|a|^2)^{2+\sigma}}\,(1-|\varphi_ a(z)|^2)^{2+\sigma} dA(z)
\\
&=C\!\int_{\D}|f'(z)|^2 \,
\frac{(1-|z|^2)^{2+\sigma}}{|1-\bar{a}z|^{b}}\,dA(z).
\end{split}
\end{displaymath}
\end{proof}
\begin{proof}[Proof of Proposition \ref{PH1}]
Firstly we  deal with the case $p\ge 2$. Note that, for $f\in H^{\infty}$  (the algebra of all bounded analytic functions on $\D$, a dense subset of $\Da$) and $u$ analytic, one has $\Delta_ uf=u\bar{f}-\widetilde{P_{\alpha}}(u\bar{f})$, where $\widetilde{P_{\alpha}}$ denotes the Bergman projection from $L^2(\D,dA_{\alpha})$ to $A^2_{\alpha}$. Therefore, $\Delta_{u} f $ is the solution of the equation $\overline{\partial} v=u\overline{f'}$ with minimal $L^2(\D,dA_{\alpha})$ norm. Now, it is well known that the solution of $\overline{\partial} v=u\overline{f'}$ given by
$$v(z)=\int_{\D}
\frac{(u\overline{f'})(w)\,(1-|w|^2)^{1+\alpha}}{(z-w)(1-\bar{w}z)^{1+\alpha}}\,dA(w)$$
satisfies the estimate
\begin{displaymath}
\int_{\D} |v(z)|^2\,dA_{\alpha}(z)\le C
\int_{\D}|(u \overline{f'})(z)|^2\,(1-|z|^2)^{2+\alpha}\,dA(z).
\end{displaymath}
Indeed, the estimate in question follows from Cauchy-Schwarz inequality and the fact that, for $c>0$ and $t>-1$, the integral $\int_{\D}
\frac{(1-|w|^2)^t\,dA(w)}{|z-w|\,|1-\bar{w}z|^{1+t+c}}$ is
comparable to $(1-|z|^2)^{-c}$ (this is just a variant of Lemma \ref{Ict}). Taking all of this into account, we obtain that
\begin{equation}\label{Es7d}
\|\Delta_ u f \|^2_{L^2(\D,dA_{\alpha})}\le C \int_{\D}|u(z)\,f'(z)|^2\,(1-|z|^2)^{2+\alpha}\,dA(z).
\end{equation}
From this inequality, it follows easily that the operator $\Delta_ u$ is bounded (or compact) if $\sup_ {z\in \D}(1-|z|)|u(z)|<\infty$ (or if $\lim_{|z|\rightarrow 1^{-}} (1-|z|)|u(z)|=0$), and it is clear that these conditions are implied by the fact that $u\in A^p_{p-2}$. Now, let
 $\{e_ n\}$ be any orthonormal set in $\Da$. Therefore, using \eqref{Es7d}, H\"{o}lder's inequality, \eqref{eqRK1} and \eqref{Id2}, we obtain
\begin{displaymath}
\begin{split}
\sum_ n \|\Delta_ u e_ n \|^p_{L^2(\D,dA_{\alpha})}&\le C \sum_ n \left ( \int_{\D}|u(z)\,e_ n'(z)|^2\,(1-|z|^2)^{2+\alpha}\,dA(z)\right )^{p/2}
\\
&
\le C \sum_ n \int_{\D}|u(z)|^p\,|e_ n'(z)|^2\,(1-|z|^2)^{p+\alpha}\,dA(z)
\\
&
=C \int_{\D}|u(z)|^p\,\sum_ n |e_ n'(z)|^2\,(1-|z|^2)^{p+\alpha}\,dA(z)
\\
& \le C \|u\|^p_{A^p_{p-2}}.
\end{split}
\end{displaymath}
A different proof for the case $p=2$ (that can be adapted to the case $p>2$) can be given as follows. Let $\{e_ n\}$ be any orthonormal basis of $\Da$. Take $0<\varepsilon<1$. Then, Lemma \ref{LS1} yields
\begin{equation}\label{ES7HS}
\begin{split}
 |\Delta_ u e_ n(w)|^2
&\le  \!\left (\int_{\D} \frac{|u(z)|^2 dA_{\alpha+\varepsilon}(z)}{|1-\bar{w}z|^{2+\alpha}}\right)\!\left (\int_{\D} \frac{|e_ n(w)-e_ n(z)|^2dA_{\alpha-\varepsilon}(z)}{|1-\bar{w}z|^{2+\alpha}}\right)
\\
& \le C \left (\int_{\D} \frac{|u(z)|^2\,dA_{\alpha+\varepsilon}(z)}{|1-\bar{w}z|^{2+\alpha}}\right)\left (\int_{\D} \frac{|e'_ n(z)|^2\,dA_{2+\alpha-\varepsilon}(z)}{|1-\bar{w}z|^{2+\alpha}}\right).
\end{split}
\end{equation}
Therefore, using \eqref{eqRK1} and Lemma \ref{Ict}, we get
\begin{displaymath}
\begin{split}
\sum_ n &\|\Delta_ u e_ n\|^2=\sum_ n \int_{\D} |\Delta_ u e_ n(w)|^2 \,dA_{\alpha}(w)
\\
&\le C \int_{\D} \left (\int_{\D} \frac{|u(z)|^2\,dA_{\alpha+\varepsilon}(z)}{|1-\bar{w}z|^{2+\alpha}}\right)\left (\int_{\D} \frac{\sum_ n |e'_ n(z)|^2\,dA_{2+\alpha-\varepsilon}(z)}{|1-\bar{w}z|^{2+\alpha}}\right)dA_{\alpha}(w)
\\
&\le C \int_{\D} \left (\int_{\D} \frac{|u(z)|^2\,dA_{\alpha+\varepsilon}(z)}{|1-\bar{w}z|^{2+\alpha}}\right)\left (\int_{\D} \frac{dA_{-\varepsilon}(z)}{|1-\bar{w}z|^{2+\alpha}}\right)dA_{\alpha}(w)
\\
&\le C \int_{\D} \left (\int_{\D} \frac{|u(z)|^2\,dA_{\alpha+\varepsilon}(z)}{|1-\bar{w}z|^{2+\alpha}}\right)dA_{-\varepsilon}(w)
\\
&= C \int_{\D} |u(z)|^2\left (\int_{\D} \frac{dA_{-\varepsilon}(w)}{|1-\bar{w}z|^{2+\alpha}}\right)dA_{\alpha+\varepsilon}(z)\le C\,\|u\|_{A^2}^2.
\end{split}
\end{displaymath}
 For $1<p<2$, one has $A^p_{p-2}\subset A^2$. Thus, by the case we have just proved, the operator $\Delta_ u$ is Hilbert-Schmidt and, in particular, compact. By Proposition \ref{pr:gen}, a sufficient condition for $\Delta_ u$ to be in the class $\mathcal{S}_ p$ is
\begin{equation}\label{SCb}
\int_{\D} \|\Delta_ u j^{\alpha}_ z\|_{L^2(\D,dA_{\alpha})}^p \,d\lambda(z)<\infty.
\end{equation}
Now, take $0<\varepsilon<1$ with $\alpha-\varepsilon>-1$ and $p-\varepsilon p>1$. Proceeding as in \eqref{ES7HS}, and then using Lemma \ref{LI2} we obtain
\begin{displaymath}
\begin{split}
|(\Delta_ u J^{\alpha}_ z)(w)|^2 &\le C \left (\int_{\D} \frac{|u(\zeta)|^2\,dA_{\alpha+\varepsilon}(\zeta)}{|1-\bar{w}\zeta|^{2+\alpha-\varepsilon}}\right)\left (\int_{\D} \frac{|(J^{\alpha}_ z)'(\zeta)|^2\,dA_{2+\alpha-\varepsilon}(\zeta)}{|1-\bar{w}\zeta|^{2+\alpha+\varepsilon}}\right)
\\
&\le C \left (\int_{\D} \frac{|u(\zeta)|^2\,dA_{\alpha+\varepsilon}(\zeta)}{|1-\bar{w}\zeta|^{2+\alpha-\varepsilon}}\right)\left (\int_{\D} \frac{dA_{2+\alpha-\varepsilon}(\zeta)}{|1-\bar{z}\zeta|^{4+2\alpha}\,|1-\bar{w}\zeta|^{2+\alpha+\varepsilon}}\right)
\\
&\le C \left (\int_{\D} \frac{|u(\zeta)|^2\,dA_{\alpha+\varepsilon}(\zeta)}{|1-\bar{w}\zeta|^{2+\alpha-\varepsilon}}\right) \frac{(1-|z|^2)^{-\alpha-\varepsilon}}{|1-\bar{w}z|^{2+\alpha+\varepsilon}}.
\end{split}
\end{displaymath}
This, together with Lemma \ref{LI2}, gives
\begin{displaymath}
\begin{split}
\|\Delta_ u & J^{\alpha}_ z\|^2_{L^2(\D,dA_{\alpha})}=\int_{\D} |(\Delta_ u J^{\alpha}_ z)(w)|^2\,dA_{\alpha}(w)
\\
&\le C (1-|z|^2)^{-\alpha-\varepsilon}\!\!\int_{\D} |u(\zeta)|^2 \!\left (\int_{\D}\frac{dA_{\alpha}(w)}{|1-\bar{w}\zeta|^{2+\alpha-\varepsilon}|1-\bar{w}z|^{2+\alpha+\varepsilon}}\right)dA_{\alpha+\varepsilon}(\zeta)
\\
& \le C (1-|z|^2)^{-\alpha-2\varepsilon}\!\int_{\D} \frac{\,|u(\zeta)|^2\, dA_{\alpha+\varepsilon}(\zeta)}{|1-\bar{z}\zeta|^{2+\alpha-\varepsilon}}.
\end{split}
\end{displaymath}
Thus,
\begin{equation}\label{Ep2}
\begin{split}
\int_{\D} \|\Delta_ u & j^{\alpha}_ z\|_{L^2(\D,dA_{\alpha})}^p  d\lambda(z)=\int_{\D} \|\Delta_ u J^{\alpha}_ z\|_{L^2(\D,dA_{\alpha})}^p \|J^{\alpha}_ z\|_{\Da}^{-p}\,d\lambda(z)
\\
&=\int_{\D} \|\Delta_ u J^{\alpha}_ z\|_{L^2(\D,dA_{\alpha})}^p (1-|z|^2)^{(2+\alpha)\frac{p}{2}-2}\,dA(z)
\\
&\le C \!\int_{\D}\left ( \int_{\D} \frac{\,|u(\zeta)|^2\, dA_{\alpha+\varepsilon}(\zeta)}{|1-\bar{z}\zeta|^{2+\alpha-\varepsilon}}\right )^{\frac{p}{2}}(1-|z|^2)^{p-2-\varepsilon p}\,dA(z).
\end{split}
\end{equation}
Now, consider an $r$-lattice $\{a_ n\}$ with associated hyperbolic disks $\{D_ n\}$. Since $p/2\le 1$ we have
\begin{displaymath}
\begin{split}
\left ( \int_{\D} \frac{\,|u(\zeta)|^2\, dA_{\alpha+\varepsilon}(\zeta)}{|1-\bar{z}\zeta|^{2+\alpha-\varepsilon}}\right )^{\frac{p}{2}}
&\le \left ( \sum_ n \int_{D_ n} \frac{\,|u(\zeta)|^2\, dA_{\alpha+\varepsilon}(\zeta)}{|1-\bar{z}\zeta|^{2+\alpha-\varepsilon}}\right )^{\frac{p}{2}}
\\
&\le C\left ( \sum_ n \frac{(1-|a_ n|^2)^{\alpha+\varepsilon}}{|1-\bar{z}a_ n|^{2+\alpha-\varepsilon}}\int_{D_ n} |u(\zeta)|^2\, dA(\zeta)\right )^{\frac{p}{2}}
\\
&\le C\sum_ n \frac{(1-|a_ n|^2)^{(\alpha+\varepsilon)\frac{p}{2}}}{|1-\bar{z}a_ n|^{p+(\alpha-\varepsilon)\frac{p}{2}}}\left (\int_{D_ n} |u(\zeta)|^2\, dA(\zeta)\right )^{\frac{p}{2}}.
\end{split}
\end{displaymath}
Putting this into \eqref{Ep2} and applying Lemma \ref{Ict}, we obtain
\begin{displaymath}
\begin{split}
\int_{\D} \|\Delta_ u & j^{\alpha}_ z\|_{L^2(\D,dA_{\alpha})}^p  d\lambda(z)
\\
& \!\le C \!\sum_ n (1-|a_ n|^2)^{(\alpha+\varepsilon)\frac{p}{2}}
\left (\!\int_{D_ n} \!\! |u(\zeta)|^2\, dA(\zeta)\!\right )^{\frac{p}{2}}\!\!\int_{\D}\!\frac{(1-|z|^2)^{p-2-\varepsilon p}dA(z)}{|1-\bar{z}a_ n|^{p+(\alpha-\varepsilon)\frac{p}{2}}}
\\
& \le C \sum_ n \left (\int_{D_ n}  |u(\zeta)|^2\, dA(\zeta)\right )^{\frac{p}{2}} \le C \|u\|_{A^p_{p-2}}^p
\end{split}
\end{displaymath}
due to Theorem 0 of \cite{AS}. This establishes \eqref{SCb} completing the proof.
\end{proof}

\subsection{Multiplication operators}
It is well known that the multiplication operator $M_{g'}:\Da\rightarrow A^2_{\alpha}$ is bounded or compact if and only if
 $M_{g''}:\Da\rightarrow A^2_{2+\alpha}$ is bounded or compact. Thus, a natural question arises here: It is true that $M_{g'}:\Da\rightarrow A^2_{\alpha}$ is in the Schatten class $\mathcal{S}_ p$ if and only if
 $M_{g''}:\Da\rightarrow A^2_{2+\alpha}$ belongs to $S_
 p$? We are going to see that this happens when $p>1$, but the result is false for $p=1$.
 Let us consider the spaces $\dot{A}^2_{\alpha}=\left\{f\in A^2_\alpha:\,f(0)=0\right\}$ and $\dot{\mathcal{D}}_{\alpha}=\left\{f\in\mathcal{D}_\alpha:\,f(0)=0\right\}$.
\begin{thm}\label{eqSp}
Let $\alpha\ge 0$, $1<p<\infty$ and $g\in H(\D)$. The following are equivalent:
\begin{enumerate}
\item[(a)] $M_{g'}:\dot{\mathcal{D}}_{\alpha}\rightarrow \dot{A}^2_{\alpha}$ is in  $\mathcal{S}_ p$;
\item[(b)] $M_{g''}:\dot{\mathcal{D}}_{\alpha}\rightarrow A^2_{2+\alpha}$ is in  $\mathcal{S}_ p$.
\end{enumerate}
\end{thm}

Taking into account Theorems \ref{th:1} and \ref{th:2}, the next result shows that it is no longer true that $M_{g'}$ being in the trace class $\mathcal{S}_ 1$ is equivalent to $M_{g''}$ being in the trace class. We recall that $g\in B_ 1$ if $g\in H(\D)$ and  $$\int_{\D} |g''(z)|\,\ \!\! dA(z)<\infty.$$
 \begin{thm} \label{FP}
 Let $g\in H(\D)$. Then,
 \begin{enumerate}
 \item[(a)] For $\alpha>0$, $M_{g''}\in \mathcal{S}_ 1 (\Da,A^2_{2+\alpha})$ if and only if $g\in B_ 1$.
 \item[(b)] If $M_ {g''}\in \mathcal{S}_ 1 (\mathcal{D},A^2_ 2)$ then $g\in B_ 1$.
 \item[(c)]  If $$\int_{\D} |g''(z)|\,\Big (\log \frac{e}{1-|z|^2}\Big
 )^{1/2} \!\! dA(z)<\infty,$$ then $M_{g''}\in \mathcal{S}_ 1 (\mathcal{D},A^2_ 2)$.
 \item[(d)] Neither of the two previous implications (b) and (c) can be reversed. Moreover, there is a function $g\in H(\D)$ with $M_{g''}\in \mathcal{S}_ 1 (\mathcal{D},A^2_ 2)$ such that
 $$\int_{\D} |g''(z)|\,\varphi(z)\, dA(z)=\infty$$ for any function $\varphi(r)$ increasing continuously to $\infty$ on $(0,1)$.
\end{enumerate}
 \end{thm}
\par One should compare Theorem \ref{FP} with the results obtained in Theorem 8 of \cite{ARSW2}, where trace class bilinear Hankel forms on the Dirichlet space are studied.
\begin{proof}[Proof of Theorem \ref{eqSp}]
We recall that if (a) or (b) holds, then $g\in B_ p$.
We first deal with the case $p\ge 2$. Since $\|f\|_{{A}^2_{\alpha}}\asymp
\|f'\|_{A^2_{2+\alpha}}$ for $f\in \dot{A}^2_{\alpha}$, then, for any orthonormal set $\{e_ n\}$ of $\dot{\mathcal{D}}_{\alpha}$, we have
\begin{equation}\label{EqMg1}
\sum_ n \|M_ {g'} e_ n\|_{{A}^2_{\alpha}}^p \asymp
\sum_ n \|(M_ {g'} e_ n)'\|_{A^2_{2+\alpha}}^p.
\end{equation}
Note that

\begin{equation}\label{EqMg2}
\begin{split}
\sum_ n \|(M_ {g'} e_ n)'\|_{A^2_{2+\alpha}}^p&\le C \left (\sum_ n \|M_ {g''} e_ n\|_{A^2_{2+\alpha}}^p+\sum_ n \|M_ {g'} e'_ n\|_{A^2_{2+\alpha}}^p\right),
\\
& \text{and}
\\
\sum_ n \|M_ {g''} e_ n\|_{A^2_{2+\alpha}}^p &\le C \left (\sum_ n \|(M_ {g'} e_ n)'\|_{A^2_{2+\alpha}}^p+\sum_ n \|M_ {g'} e'_ n\|_{A^2_{2+\alpha}}^p\right).
\end{split}
\end{equation}
Since $g\in B_ p$, it follows from H\"{o}lder's inequality that
\begin{displaymath}
\begin{split}
\sum_ n \|M_ {g'} e'_ n\|_{A^2_{2+\alpha}}^p&=\sum_ n \left (\int_{\D} |g'(z)|^2\,|e'_ n(z)|^2\,dA_{2+\alpha}(z)\right)^{p/2}
\\
&\le C \sum_ n \int_{\D} |g'(z)|^p\,|e'_ n(z)|^2\,(1-|z|^2)^{p+\alpha}\,dA(z)
\\
&\le C  \int_{\D} |g'(z)|^p\|J^{\alpha}_z\|_{{\mathcal{D}}^2_{\alpha}}^2\,(1-|z|^2)^{p+\alpha}\,dA(z)
\\
&=C  \int_{\D} |g'(z)|^p\,(1-|z|^2)^{p-2}\,dA(z)\le C \|g\|_{B_ p}^p.
\\
\end{split}
\end{displaymath}
From this, \eqref{EqMg1} and \eqref{EqMg2}, it is easy to see that (a) and (b) are equivalent.
\par Now we deal with the case $1<p<2$. Since $\dot{A}^2_{\alpha}$ coincides with ${\dot{\mathcal{D}}}_{2+\alpha}$ with equivalent norms, we will see that (b) is equivalent to $M_{g'}:\dot{\mathcal{D}}_{\alpha} \rightarrow {\dot{\mathcal{D}}}_{2+\alpha}$ being in $\mathcal{S}_ p$. For all orthonormal sets $\{e_ n\}$ of $\dot{\mathcal{D}}_{\alpha}$ and $\{f_ n\}$ of ${\dot{\mathcal{D}}}_{2+\alpha}$, we have that
$$\sum_ n |\langle M_{g'}e_ n,f_ n\rangle_{{{\mathcal{D}}}_{2+\alpha}}|^p\asymp  (II),$$
where
$$(II)=\sum_ n |\langle (M_{g'}e_ n)',f'_ n\rangle_{A^2_{2+\alpha}}|^p.$$

We see that $(II)\le C \left ((IIa)+(IIb)\right)$, where
$$(IIa)=\sum_ n \left|\int_{\D} g'(z)\, e'_ n(z)\,\overline{f'_ n(z)} (1-|z|^2)^2\,dA_{\alpha}(z)\right |^p,$$
and
$$(IIb)=\sum_ n \left|\int_{\D} g''(z)\, e_ n(z)\,\overline{f'_ n(z)} (1-|z|^2)^2\,dA_{\alpha}(z)\right |^p.$$
Next, we are going to see that the term $(IIa)$ is dominated by the $B_ p$ norm of $g$. Indeed, H\"{o}lder's inequality gives
\begin{displaymath}
\begin{split}
(IIa) &\le \sum_ n \left (\int_{\D} |g'(z)|\,|e'_ n(z)|\,|f'_ n(z)|\,(1-|z|^2)^2\,dA_{\alpha}(z) \right )^p
\\
&\le\left ( \sum_ n \int_{\D} |g'(z)|^p \,|e'_ n(z)|^p\,|f'_ n(z)|^{2-p}\,(1-|z|^2)^{2}\,dA_{\alpha}(z)\right) \|f_ n\|_{{\mathcal{D}}_{2+\alpha}}^{p-1}
\\
&\le \int_{\D} |g'(z)|^p \left (\sum_ n|e'_ n(z)|^p\,|f'_ n(z)|^{2-p}\right)(1-|z|^2)^{2}\,dA_{\alpha}(z).
\end{split}
\end{displaymath}
At this point, we use H\"{o}lder's inequality again together with \eqref{eqRK1} and \eqref{Id2} to obtain that
\begin{displaymath}
\begin{split}
(IIa) & \le \int_{\D} |g'(z)|^p \Big (\sum_ n|e'_ n(z)|^2 \Big)^{p/2} \Big (\sum_ n|f'_ n(z)|^2 \Big)^{\frac{2-p}{2}}(1-|z|^2)^{2}\,dA_{\alpha}(z)
\\
& \le \int_{\D} |g'(z)|^p \,\|J^{\alpha}_ z \|_{\mathcal{D}_{\alpha}}^p\,\|J^{2+\alpha}_ z\|_{\mathcal{D}_{2+\alpha}}^{2-p}\,(1-|z|^2)^{2}\,dA_{\alpha}(z)
\\
& =\int_{\D}|g'(z)|^p \,(1-|z|^2)^{-(2+\alpha)\frac{p}{2}}\,(1-|z|^2)^{-(4+\alpha)\frac{2-p}{2}}\,(1-|z|^2)^{2} \,dA_{\alpha}(z)
\\
&\le C \|g\|_{B_ p}^p.
\end{split}
\end{displaymath}
Thus, putting all together, we see that if $g\in B_ p$ then
\begin{displaymath}
\sum_ n |\langle M_{g'}e_ n,f_ n\rangle_{{{\mathcal{D}}}_{2+\alpha}}|^p<\infty \Leftrightarrow \sum_ n  |\langle M_{g''}e_ n,f'_ n\rangle_{A^2_{2+\alpha}} |^p<\infty.
\end{displaymath}
Observe that we have just proved one implication, but the other is proved exactly in the same way.
Finally, it is clear that  $\{f_n\}$ is an orthonormal set of ${\dot{\mathcal{D}}}_{2+\alpha}$ if and only if $\{f'_n\}$ is an orthonormal set of $A^{2}_{2+\alpha}$, which gives  $(a) \Leftrightarrow (b)$.
\end{proof}

 \begin{proof}[Proof of Theorem \ref{FP}]
 If the decomposition of the positive operator
 $M_{g''}^*M_{g''}$ is given by
$\sum_ n \lambda_ n \langle \cdot,e_
n\rangle_{\Da}\, e_ n,$ then, as in the proof of Proposition \ref{nece},  $\{e_ n\}$ is an orthonormal
basis of $\Da$. Thus, by Lemma \ref{le:2}, we have that
$$(1-|z|^2)^{-1-\alpha}\le C\,\sum_ n |e_ n(z)|\,|e_ n'(z)|,$$
and we deduce
\begin{displaymath}
\begin{split}
\int_{\D} |g''(z)|\,dA(z)&\le C \int_{\D} |g''(z)| (1-|z|^2)\left
(\sum_ n |e_ n(z)|\,|e'_ n(z)|\right )\,dA_{\alpha}(z)
\\
&=C\sum_ n \int_{\D} |g''(z)| (1-|z|^2)\, |e_ n(z)|\,|e'_
n(z)|\,dA_{\alpha}(z)
\\
&\le C\sum_ n \left (\int_{\D} |g''(z)|^2\,|e_
n(z)|^2\,(1-|z|^2)^2\,dA_{\alpha}(z)\right )^{1/2}
\\
&=C\sum_ n \|M_{g''}e_ n\|_{A^2_{2+\alpha}}
=C\sum_ n \big(\langle M_{g''}e_ n,M_{g''}e_ n \rangle _{A^2_{2+\alpha}}\big )^{1/2}
\\
&=C\sum_ n |\lambda_ n|^{1/2}=C \|M_{g''}\|_{\mathcal{S}_ 1},
\end{split}
\end{displaymath}
which gives (b) and the necessity in (a).

Now we proceed to show part (c), and the sufficiency in (a). Bearing in mind \eqref{eqRK1} and \eqref{RKBergman}, for all orthonormal sets  $\{e_ n\}$ of $\Da$ and  $\{f_
n\}$ of $A^2_{2+\alpha}$, we have
 \begin{equation}\label{S7Eq1}
\begin{split}
\|M_{g''}\|_{\mathcal{S}_ 1}&\le \sum_ n \big |\langle M_{g''}e_ n,f_ n\rangle_{A^2_{2+\alpha}}\big |
\\
& \le
 \int_{\D} |g''(z)|\left (\sum_ n |e_ n(z)|\,|f_
n(z)|\right ) \,dA_{2+\alpha}(z)
\\
& \le \int_{\D} |g''(z)|\left (\sum_ n |e_ n(z)|^2\right
)^{\!1/2}\!\!\left (\sum_ n |f_ n(z)|^2\right )^{\!1/2}\!\!\! \,dA_{2+\alpha}(z)
\\
& \le \int_{\D} |g''(z)|\, \|K_ z\|_{\Da}\,\|B^{2+\alpha}_ z\|_{A^2_
{2+\alpha}}\,dA_{2+\alpha}(z)
\\
&\asymp \int_{\D} |g''(z)|\,\|K_ z\|_{\Da}\, dA_{\alpha/2}(z)
\end{split}
 \end{equation}
 and, according to \eqref{RKDA}, this is comparable to
 $$\int_{\D} |g''(z)|\,\left (\log \frac{e}{1-|z|^2}\right
 )^{1/2} \!\! dA(z)\qquad \textrm{if}\quad \alpha=0$$
 establishing part (c); and is comparable to
 $$\int_{\D} |g''(z)|\,dA(z) \qquad \textrm{for}\quad \alpha>0$$ that gives the remaining part in (a).

 Now we prove (d). To see that part (b) can not be reversed, consider the functions $g_ a(z)=(1-\bar{a}z)^{-\gamma}$, $\gamma>0$ and $a\in \D$. Then, the same argument leading to \eqref{gabplog3} yields
 \begin{equation}\label{S7Eq2}
 \int_{\D} |g_ a''(z)|\,\left (\log \frac{e}{1-|z|^2}\right
 )^{1/2} \!\! dA(z)\asymp (1-|a|^2)^{-\gamma}\,\left (\log \frac{e}{1-|a|^2}\right
 )^{1/2},
 \end{equation}
 and by what we have just proved (see \eqref{S7Eq1} and the comments after that), one gets
 $$\|M_{g_ a''}\|_{\mathcal{S}_ 1(\mathcal{D},A^2_ 2)} \le C (1-|a|^2)^{-\gamma}\,\left (\log \frac{e}{1-|a|^2}\right
 )^{1/2}.$$
 On the other hand, we can estimate the trace of $M_{g_ a''}$ from below as follows. Let $\{e_ n\}$ be any orthonormal basis of $\mathcal{D}$. Then
 \begin{displaymath}
 \begin{split}
 \|M_{g_ a''}\|_{\mathcal{S}_ 1(\mathcal{D},A^2_ 2)}&\ge \|M_{g_ a''}\|_{\mathcal{S}_ 2(\mathcal{D},A^2_ 2)} =\left (\sum_ n \|M_{g_ a''} e_ n\|_{A^2_ 2}^2\right )^{1/2}
 \\
 &=\left (\sum_ n \int_{\D} |g_ a''(z)|^2 \,|e_ n(z)|^2 \,dA_ 2(z) \right )^{1/2}
 \\
 &=\left ( \int_{\D} |g_ a''(z)|^2 \log \frac{e}{1-|z|^2}\ \,dA_ 2(z) \right )^{1/2}
 \\
 &\ge C (1-|a|^2)^{-\gamma}\,\Big (\log \frac{e}{1-|a|^2}\Big
 )^{1/2}.
 \end{split}
 \end{displaymath}
 All together yields
 $$\|M_{g_ a''}\|_{\mathcal{S}_ 1(\mathcal{D},A^2_ 2)} \asymp (1-|a|^2)^{-\gamma}\,\Big (\log \frac{e}{1-|a|^2}\Big
 )^{1/2}.$$
 Now, from that and \eqref{S7Eq2}, we see that the sufficient condition in part (c) is sharp in a certain sense. Also, since $\|g_ a\|_{B_ 1}\asymp (1-|a|^2)^{-\gamma}$, we see that part (b) can not be reversed.
\par  To see that part (c) can not be reversed, consider a lacunary series $g(z)=\sum_{k=0}^{\infty} a_ k z^{n_ k}$ with $n_{k+1}/n_ k\ge c>1$. Now, we claim that, if $\sum_ k n_ k |a_ k|<\infty$ then $M_{g''}$ belongs to the trace class $\mathcal{S}_ 1(\mathcal{D},A^2_{2})$. Indeed,
 \begin{equation}\label{Mtrace}
 \|M_{g''}\|_{\mathcal{S}_ 1}=\Big\|\sum_ {k=1}^{\infty} n_ k(n_ k-1) a_ kM_{z^{n_{k}-2}}\Big\|_{\mathcal{S}_ 1}\le C\sum_ k n_ k^{2}\, |a_ k| \,\|M_{z^{n_ k-2}}\|_{\mathcal{S}_ 1}.
 \end{equation}
 Since $\{e_ n\}_{n\ge 0}=\big \{\frac{z^n}{(n+1)^{1/2}}\big \}_{n\ge 0}$ is an orthonormal basis of $\mathcal{D}$ and $\{\sigma_ n\}_{n\ge 0}=\{c_ n z^n\}_{n\ge 0}$ with $c_ n\asymp (n+1)^{3/2}$ is an orthonormal basis of $A^2_ 2$, an easy computation gives that, for $f(z)=\sum_{n=0}^{\infty} c_ n z^n$, the multiplication operator $M_{z^j}$ has the decomposition
 $$M_{z^j}(f)=\sum_{n=j}^{\infty}\frac{1}{c_ n (n-j+1)^{1/2}}\, \langle f, e_{n-j}  \rangle _{\mathcal{D}}\,\sigma_ n,$$ and therefore, the singular numbers $\lambda_ n$ satisfy that $\lambda_ n \asymp \sqrt{\frac{1}{(n+1)^3 (n-j+1)}}$. Thus,
 $$\|M_{z^j}\|_{\mathcal{S}_ 1}=\sum_{n\ge j} \lambda_ n \asymp \sum_{n\ge j} \sqrt{\frac{1}{(n+1)^3 (n-j+1)}}\,\asymp\, j^{-1},$$
 and putting this into \eqref{Mtrace} gives
 $$\|M_{g''}\|_{\mathcal{S}_ 1} \le C \sum_ k  n_ k  |a_ k|.$$
This shows together  with part $(b)$ that given a lacunary series $g(z)=\sum_ k a_ k z^{n_ k}$, the multiplication operator $M_{g''}:\mathcal{D}\rightarrow A^2_ 2$ belongs to $\mathcal{S}_ 1$ if and only if $\sum_ k n_ k |a_ k|<\infty$, and it is well known that this condition is equivalent to $g$ being in $B_ 1$
\cite[p. $100$]{Zhu}.
\par Now, given a function $\varphi$ as described in part (d), it is straightforward to select the numbers $\{a_ k\} $ and the sequence $\{n_ k\}$ so that the summability condition $\sum_ k n_ k |a_ k|<\infty$ is met, but $\int_{\D} |g''|\,\varphi \,dA=\infty$.
\end{proof}

\textbf{Acknowledgments:} The authors would like to thank the referee for his/her
comments and suggestions that improved the final version of the paper, and also for pointing out  the relation between the integration operator $T_ g$  and the previous work of A. Calder\'{o}n.

% ------------------------------------------------------------------------

\end{document}